\documentclass[11pt,a4paper]{article}
\usepackage[toc,page]{appendix}

\usepackage[all,cmtip]{xy}
\entrymodifiers={+!!<0pt,\fontdimen22\textfont2>}

\usepackage{fouriernc}
\usepackage[T1]{fontenc}
\usepackage[english]{babel}         
\usepackage{verbatim}               

\usepackage{amsmath}
\usepackage{amsfonts}
\usepackage{amssymb}
\usepackage{mathrsfs}    
\usepackage{amsthm}
\usepackage{xcolor}

\usepackage{hyperref}           

\usepackage{cite}
\usepackage{stmaryrd}

\usepackage{graphicx}
\usepackage{pgf}
\usepackage{eepic}
\usepackage{pstricks}
\usepackage[all,cmtip]{xy}

\newcommand{\spf}[1]{\mathrm{Spf}\left(#1\right)}

\newcommand{\weak}[1]{\langle #1\rangle^\dagger}

\usepackage{epsfig}   

\newcommand{\rlser}[1]{\{\!\{#1\}\!\}}
\newcommand{\rpow}[1]{\{#1\}}

\theoremstyle{plain}
\newtheorem*{thrm}{Theorem}
\newtheorem{theo}{Theorem}[section]

\newtheorem{prop}[theo]{Proposition}
\newtheorem{cor}[theo]{Corollary}

\newtheorem{lem}[theo]{Lemma}

\newtheorem{guess[theo]}{Guess}
\theoremstyle{definition}

\newtheorem{defn}[theo]{Definition}

\theoremstyle{remark}
\newtheorem{rem}[theo]{Remark}

\newtheorem*{claim}{Claim}

\newcommand{\proofof}[1]{\end{#1}\begin{proof}}

\makeatletter
\renewcommand\section{\@startsection {section}{1}{\z@}%
  {-3.5ex \@plus -1ex \@minus -.2ex}{2.3ex \@plus.2ex}%
  {\normalfont\large\bfseries}}
\renewcommand\subsection{\@startsection{subsection}{2}{\z@}%
  {-3.25ex\@plus -1ex \@minus -.2ex}{1.5ex \@plus .2ex}%
  {\normalfont\bfseries}}

\newcommand{\sh}[1]{\mathcal{#1}}

\newcommand{\Z}{{\mathbb Z}}

\renewcommand{\P}{{\mathbb P}}
\newcommand{\A}{{\mathbb A}}

\DeclareMathAlphabet{\mathrmsl}{OT1}{cmr}{m}{sl}

\newcommand{\rssymb}[2]{\newcommand{#1}{\mathrmsl{#2}} }
\newcommand{\oper}[3][n]{\newcommand{#2}{\mathop{\mathrm{#3}}%
\ifx n#1\nolimits\else\limits\fi} }
\newcommand{\rsoper}[3][n]{\newcommand{#2}{\mathop{\mathrmsl{#3}}%
\ifx n#1\nolimits\else\limits\fi} }

\newcommand{\lser}[1]{(\!(#1)\!)}
\newcommand{\pow}[1]{\llbracket #1 \rrbracket}

\newcommand{\spec}[1]{\mathrm{Spec}\left(#1\right)}
\newcommand{\cur}[1]{\mathcal{#1}}
\newcommand{\Norm}[1]{\left\Vert #1\right\Vert}
\newcommand{\norm}[1]{\left\vert#1\right\vert}

\newcommand{\rig}{\mathrm{rig}}
\newcommand{\ek}{\cur{E}_K}
\newcommand{\ekd}{\cur{E}_K^\dagger}
\newcommand{\tate}[1]{\langle #1 \rangle}

\newcommand{\pn}{(\varphi,\nabla)}

\oper\Ad{Ad}

\oper\val{val}
\oper\coker{coker}
\oper\mult{mult}
\oper\Iso{Iso}
\oper\End{End}
\oper\Aut{Aut}
\oper\Sub{Sub}
\oper\Alt{Alt}
\oper\Ext{Ext}
\oper\Pic {Pic}
\oper\Sym{Sym}
\oper\Spec{Spec}
\oper\Spf{Spf}
\oper\Sp{Sp}
\oper\Spa{Spa}
\oper\Proj{Proj}
\rsoper\divg{div}
\rsoper{\sym}{sym}
\rsoper{\alt}{alt}
\rsoper\trace{tr}
\rssymb\id{id}

\newcommand{\thismonth}{\ifcase\month\or
  January\or February\or March\or April\or May\or June\or
  July\or August\or September\or October\or November\or December\fi
  \space\number\year}

\title{Rigid cohomology over Laurent series fields II: Finiteness and Poincar\'{e} duality for smooth curves}
\author{Christopher Lazda and Ambrus P\'{a}l}

\begin{document}

\maketitle 

\abstract{In this paper we prove that the $\ekd$-valued cohomology, introduced in \cite{rclsf1} is finite dimensional for smooth curves over Laurent series fields $k\lser{t}$ in positive characteristic, and forms an $\ekd$-lattice inside `classical' $\ek$-valued rigid cohomology. We do so by proving a suitable version of the $p$-adic local monodromy theory over $\ekd$, and then using an \'{e}tale pushforward for smooth curves to reduce to the case of $\A^1$. We then introduce $\ekd$-valued cohomology with compact supports, and again prove that for smooth curves, this is finite dimensional and forms an $\ekd$-lattice in $\ek$-valued cohomology with compact supports. Finally, we prove Poincar\'{e} duality for smooth curves, but with restrictions on the coefficients.}

\tableofcontents

\section*{Introduction}\addcontentsline{toc}{section}{Introduction}

This is the second in a series of papers \cite{rclsf1,rclsf3} dedicated to the construction of a new $p$-adic cohomology theory for varieties over local fields of positive characteristic. A detailed introduction to the whole series is given in \cite{rclsf1}, so here we will give a brief overview of the results contained in this paper.

In the first paper \cite{rclsf1} we introduced a version of rigid cohomology for varieties over the Laurent series field $k\lser{t}$ with values in vector spaces over the bounded Robba ring $\ekd$ (here $K$ is a complete discretely valued field of characteristic $0$ with residue field $k$). There we proved that the cohomology groups were welll-defined and functorial, as well as introducing categories of coefficients. The main result in this paper is that base change holds for smooth curves over $k\lser{t}$, a precise statement of which is as follows. 

\begin{thrm}[\ref{finitecurves}] Let $X/k\lser{t}$ be a smooth curve and $\sh{E}\in F\text{-}\mathrm{Isoc}^\dagger(X/\ekd)$ an overconvegent $F$-isocrystal, with associated overconvergent $F$-isocrystal $\hat{\sh{E}}\in F\text{-}\mathrm{Isoc}^\dagger(X/\ek)$. Then the base change morphism
$$ H^i_\rig(X/\ekd,\sh{E})\otimes_{\ekd} \ek \rightarrow H^i_\rig(X/\ek,\hat{\sh{E}})
$$
is an isomorphism.
\end{thrm}

Here $\cur{E}_K$ is the Amice ring, which is the $p$-adic completion of $\ekd$. This implies that $H^i_\rig(X/\ekd,\sh{E})$ is finite dimensional, of the expected dimension, and that the linearised Frobenius morphism is bijective. The method of proof is very similar to that used by Kedlaya in \cite{kedlayafiniteness} to prove finite dimensionality of `classical' rigid cohomology. We first prove a version of the $p$-adic local monodromy theorem, exploiting the fact that $\ekd$ can be viewed as a kind of `dagger algebra' over $K$ to adapt Kedlaya's proof of a monodromy theorem for dagger algebras in \emph{loc. cit.} to our situation. This will then more or less immediately imply the required result for $\A^1_{k\lser{t}}$. 

We then use \'{e}tale pushforward - the point is that locally any smooth curve admits a finite \'{e}tale map to $\A^1_{k\lser{t}}$, and after making a finite separable extension of $k\lser{t}$ we can lift this to characteristic zero to construct an \'{e}tale pushforward functor and thus reduce to the case of the affine line. The construction is slightly more difficult than in classical rigid cohomology, since one must first choose models over $k\pow{t}$ before lifting. It is not clear to us whether a similar construction can be made in higher dimensions, and it is for this reason that we are compelled to restrict to the case of curves.

We also introduce a version of $\ekd$-valued rigid cohomology with compact supports, again with and without coefficients, and use similar methods as before to show that for smooth curves, these groups are finite dimensional, and form an $\ekd$-lattice inside $\ek$-valued rigid cohomology. This allows us to deduce Poincar\'{e} duality entirely straightforwardly, by base changing to $\ek$, however, using this method forces us to make restrictions on coefficients, namely we must restrict to those $F$-isocrystals which extend to a compactification.

In the third paper in the series \cite{rclsf3} we will discuss some arithmetic applications of the theory. We will introduce a more refined category of coefficients such that the associated cohomology groups come with a natural Gauss--Manin connection, and then use this to attach $(\varphi,\nabla)$-modules over the Robba ring $\cur{R}_K$, and hence $p$-adic Weil--Deligne representations, to smooth curves over $k\lser{t}$. We will also discuss questions such as $\ell$-independence and a $p$-adic version of the weight-monodromy conjectures.

\section{A $p$-adic local monodromy theorem after Kedlaya}\label{s1ked}

For the whole of this paper, notations will be as in \cite{rclsf1}. That is, $k$ will be a field of characteristic $p>0$, $\cur{V}$ will be a complete DVR with residue field $k$ and fraction field $K$ of characteristic $0$, $\pi$ will be a uniformiser for $\cur{V}$. We let $\norm{\cdot}$ denote the norm on $K$ such that $\norm{p}=1/p$, and we let $r=\norm{\pi^{-1}}>1$. We will let $k\lser{t}$ denote the Laurent series field over $k$, and $\ekd,\cur{R}_K$ and $\ek$ respectively denote the bounded Robba ring, Robba ring, and Amice ring over $K$, and $cur{O}_{\ekd}$, $\cur{O}_{\ek}$ the valuation rings of $\ekd$ and $\ek$ respectively. For definitions of these, see the introduction to \cite{rclsf1}. We will fix a Frobenius $\sigma$ on $\cur{V}\pow{t}$, and denote by $\sigma$ the induced Frobenius on any of the rings $S_K=\cur{V}\pow{t}\otimes_\cur{V} K$, $\cur{O}_{\ek}$, $\ek$, $\cur{O}_{\ekd}$, $\ekd$, $\cur{R}_K$.

In \cite{rclsf1} we constructed, for any $k\lser{t}$-variety, a category $F\text{-}\mathrm{Isoc}^\dagger(X/\ekd)$ of overconvergent $F$-isocrystals on $X/\ekd$, as well as cohomology groups $H^i_\rig(X/\ekd,\sh{E})$ which are vector spaces over $\ekd$, functorial in both $\sh{E}$ and $X$. We will not go into the details of this construction here. For a $k\lser{t}$-variety $X$ and an overconvergent $F$-isocrystal $\sh{E}$ on $X/\ekd$, two of the fundamental results that one would want to know about $\ekd$-valued rigid cohomology are finite dimensionality and base change. This latter means that the base change map
$$
H^i_\rig(X/\ekd,\sh{E})\otimes_{\ekd}\ek\rightarrow H^i_\rig(X/\ek,\hat{\sh{E}})
$$
is an isomorphism, where $\hat{\sh{E}}$ is the associated overconvergent $F$-isocrystal on $X/\ek$ (see \S5 of \cite{rclsf1}). In this section, we pave the way for proving this for smooth curves over $k\lser{t}$ by proving a version of Kedlaya's relative local monodromy theorem. The statement of Kedlaya's result is as follows (for more details about the specific terms appearing in the statement, see \cite{kedlayafiniteness}).

\begin{theo}[\cite{kedlayafiniteness}, Theorem 5.1.3]\label{keduni} Let $A$ be an integral dagger algebra over $K$, and let $M$ be a free $(\varphi,\nabla)$-module over the relative Robba ring $\cur{R}_A$ (as defined in \S2.5 of \emph{loc. cit.}). Then there exists a weakly complete localisation $B$ of $A$, an integer $m\geq0$, a finite \'{e}tale extension $B_1$ of $B_0=B^{\sigma^{-m}}$, a finite, \'{e}tale, Galois extension $\cur{R}'$ of $\cur{R}^\mathrm{int}_{B_1}$, and a continuous $B_1$-algebra isomorphism
$$
\cur{R}_{B_1}\cong \cur{R}'':=\cur{R}_{B_1}\otimes_{\cur{R}_{B_1}^\mathrm{int}} \cur{R}'
$$
such that $M\otimes \cur{R}''$ is unipotent. 
\end{theo}

This theorem allows Kedlaya to prove generic coherence of the higher direct images of an overconvergent $F$-isocrystal along the projection $\A^1_X\rightarrow X$, for $X\cong\spec{A_0}$ a smooth affine variety over $k$, with weakly complete lift $A$. In keeping with the general philosophy that we should view $\ekd$ as a `weakly complete lift' of $k\lser{t}$, one would hope that by replacing $A$ in the statement of the above theorem by $\ekd$, one could hope to prove finite dimensionality of
$$
H^i_\rig(\A^1_{k\lser{t}}/\ekd,\sh{E})
$$
for any overconvergent $F$-isocrystal $\sh{E}$ on $\A^1_{k\lser{t}}/\ekd$. This will be our eventual strategy. Almost everything in this section is based upon Section 5 of \cite{kedlayafiniteness}, and there are very few new ideas involved. The only real insight is that the relationship between $\ekd$ and $\ek$ is exactly analogous to the relationship between a dagger algebra and the completion of its fraction field, and hence the methods used in \emph{loc. cit.} should work more or less verbatim to descend properties of $\nabla$-modules from $\cur{R}_{\ek}$ to $\cur{R}_{\ekd}$.

Our first task is to introduce the objects the will allow us to formulate our version of Kelaya's theorem, including the Robba ring over $\ekd$. We will also need to prove various important properties of this ring, and of it's `integral' subring $\cur{R}^\mathrm{int}_{\ekd}$. Recall that we have
\begin{align*}
\ekd&=\mathrm{colim}_{\eta<1}\cur{E}_\eta \\
\cur{E}_\eta &= \left\{\left.\sum_i a_it^i \in\cur{E}_K\;\right|\;  \norm{a_i}\eta^i\rightarrow 0\text{ as }i\rightarrow -\infty \right\},
\end{align*}
each $\cur{E}_\eta$ is equipped with the norm
$$
\Norm{\sum_i a_it^i}_\eta = \max \left\{\sup_{i<0} \norm{a_i}\eta^i,\sup_{i\geq 0} \norm{a_i}\right\}
$$
and these induce a direct limit topology on $\ekd$ which we call the fringe topology. Recall that for each $\eta$ we may define the Robba ring over $\cur{E}_\eta$ to be
$$ \bigcup_{s>0} \cur{R}_{\eta,s}
$$
where $\cur{R}_{\eta,s}$ is the ring of series $\sum_{i\in\Z} f_i y^i$ with $f_i\in \cur{E}_\eta$ such that $\Norm{f_i}_\eta r^{-is'}\rightarrow 0$ as $\norm{i}\rightarrow \infty$ for all $0<s'\leq s$. To define the Robba ring over $\ekd$ requires a bit more care. For each $s>0$, we define $\cur{R}_{\ekd,s}$ to be the ring consisting of series $\sum_i f_iy^i$ such that for all $0<s'\leq s$, there exists some $\eta$ such that $f_i\in\cur{E}_\eta$ and $\Norm{f_i}_\eta r^{-is'}\rightarrow 0 $ as $\norm{i}\rightarrow \infty$. 

\begin{defn} \label{robba}The Robba ring $\cur{R}_{\ekd}$ is by definition
$$\cur{R}_{\ekd}:= \bigcup_{s>0} \cur{R}_{\ekd,s},
$$
that is it consists of series $\sum_i f_iy^i$ such that for all $s>0$ sufficiently small, there exists some $\eta$ such that $f_i\in\cur{E}_\eta$ and $\Norm{f_i}_\eta r^{-is}\rightarrow 0 $
\end{defn}

Thus, if $\cur{R}_{\ek}$ denotes the Robba ring over $\ek$, in the usual sense as in Section 2 of \cite{kedlayafiniteness}, there is a natural inclusion
$$
\cur{R}_{\ekd}\rightarrow \cur{R}_{\ek}. 
$$
arising from the inclusion $\ekd\rightarrow \ek$. We will let $\cur{R}_{\ekd}^\mathrm{int}\subset \cur{R}_{\ekd}$ denote the subring consisting of series with integral coefficients, i.e. coefficients in $\cur{O}_{\ekd}$. There is thus a similar inclusion 
$$
\cur{R}^\mathrm{int}_{\ekd}\rightarrow \cur{R}^\mathrm{int}_{\ek}
$$
where $\cur{R}^\mathrm{int}_{\ek}$ is defined analogously. Note that we have $\cur{R}_{\cur{E}_\eta}\subset \cur{R}_{\ekd}$, but be warned that $\cup_\eta \cur{R}_{\cur{E}_\eta} \subsetneq \cur{R}_{\ekd}$. The former is obtained by reversing the quantifiers `$\forall$ sufficiently small $s>0$' and `$\exists \eta<1$' in the definition of the latter.

We will need to know how to `lift' certain finite extension of $k\lser{t}\lser{y}$ to those of $\cur{R}_{\ekd}$, however, this will not be achieved in an entirely straightforward manner. The first problem is that the residue field of $\cur{R}_{\ekd}$ is \emph{not} the whole of the double Laurent series field $k\lser{t}\lser{y}$, but is in fact somewhat smaller.

\begin{lem} \label{quotfieldlin} The quotient ring $\cur{R}_{\ekd}^\mathrm{int}/(\pi)$ is isomorphic to the ring of Laurent series
$$ \sum_i f_iy^i \in k\lser{t}\lser{y}
$$
such that there exist positive integers $c,d$ with $-v_t(f_i)\leq ci+d$ for all $i$. This is a subfield of $k\lser{t}\lser{y}$ which contains $k\lser{t}(y)$.
\end{lem}

\begin{proof} It is straightforward to see that $\cur{R}_{\ekd}^\mathrm{int}/(\pi)$ is contained inside $k\lser{t}\lser{y}$, since for $\sum_if_iy^i\in \cur{R}_{\ekd}^\mathrm{int}$ and $i\ll0$ we have $\Norm{f_i}\leq \Norm{f_i}_\eta < 1$, where $\Norm{\cdot}$ denote the $p$-adic norm on $\cur{O}_{\ek}$. To show that it is the ring described, first suppose that we have some $\sum_if_iy^i\in k\lser{t}\lser{y}$ such that $-v_t(f_i)\leq ci+d$ for some $c,d$. Write $f_i=\sum_jf_{ij}t^j$ and lift each non-zero $f_{ij}\in K$ to some $\tilde{f}_{ij}\in K$ of norm one. I claim firstly that $\tilde{f}_i=\sum_j\tilde{f}_{ij}t^j$ lies in $\cur{O}_{\ekd}$, and secondly that $\sum_i\tilde{f}_iy^i$ is in $\cur{R}_{\ekd}$.  

Indeed, in the first case we actually have that $\tilde{f}_i\in \cur{V}\pow{t}[t^{-1}]\subset \cur{O}_{\ekd}$, and for the second note that since $\tilde{f}_i=0$ for $i\ll0$, it suffices to check the growth condition as $i\rightarrow \infty$. But since $\tilde{f}_{ij}=0$ for $j\leq -ci-d$ we have that
$$ \Norm{\tilde{f}_{ij}}_\eta \leq \eta^{-ci-d} $$
and hence for each $s$ we can find some $\eta$ close enough to 1 to ensure that $ \eta^{-ci-d} r^{-is}\rightarrow 0 $
as $i\rightarrow \infty$. Hence $\sum_i\tilde{f}_i y^i\in \cur{R}_{\ekd}$ as required.

Conversely, let us suppose that $f=\sum_if_iy^i$ does not satisfy the growth condition, that is $-v_t(f_i)\geq ci+d$ for all integers $c,d$. For \emph{any} lift $\sum_i \tilde{f}_iy^i$ of $f$ to $\cur{O}_{\ekd}\pow{y,y^{-1}}$, then since $\tilde{f}_i$ must have a term in $t^{v_t(f_i)}$, it follows that we must have
$\Norm{\tilde{f}_i}_\eta\geq \eta^{v_t(f_i)}$
and so
$$ \Norm{\tilde{f}_i}_\eta r^{-is}\geq \eta^{-ci-d}r^{-is}
$$  
for all $c,d$. Hence for any given $s$, no matter how we choose $\eta$, we can always choose some $c$ to make this $\rightarrow \infty$ as $i\rightarrow \infty$, thus $f$ does not lift.

Finally, since the element $y$ of $k\lser{t}\lser{y}$ trivially satisfies the growth condition, to prove the final claim it suffices to show that if $f\in k\lser{t}\lser{y}$ is non-zero and satisfies the growth condition, then so does $f^{-1}$. We easily reduce to the case where $f=1+\sum_{i\geq1}a_iy^i$ with $a_i\in k\lser{t}$, let us choose $c,d$ such that $-v_t(a_i)\leq ci+d$. Then $f^{-1}=1+\sum_{i\geq1}b_iy^i$ where $b_i$ is a sum of things of the form $a_{i_1}\ldots a_{i_m}$ with $i_L\geq1$ and $i_1+\ldots+i_m=i$. Hence $$-v_t(b_i)\leq \sup_{i_1+\ldots+i_m=i} \left\{\sum_l-v_t(a_{i_l})\right\}\leq \sup_{i_1+\ldots+i_m=i} \left\{\sum_l (ci_l+d)\right\} \leq (c+d)i$$ for all $i\geq 1$ and so $f^{-1}$ satisfies the growth condition, as required.\end{proof}

Let us denote this field of `overconvergent' Laurent series by $k\lser{t}\rlser{y}$, and the part with positive $y$-adic valuation by $k\lser{t}\{y\}$. The next result tells us that with respect to totally ramified extensions, this field behaves essentially the same as the full double Laurent series field $k\lser{t}\lser{y}$. 

\begin{prop} Write $F=k\lser{t}$. The field $F\rlser{y}$ is $y$-adically Henselian, and if we have $P\in F\rpow{y}[X]$ an Eisenstein polynomial, with root $u$, so that there is an isomorphism $F\lser{y}[u]\cong F\lser{u}$, then there is an equality
$$ F\rlser{y}[u] = F\rlser{u}
$$
inside $k\lser{t}\lser{u}$. In particular, every finite, separable, (Galois) totally ramified extension $F\lser{u}/F\lser{y}$ arises from a unique finite, separable, (Galois) totally ramified extension of the form $F\rlser{u}/F\rlser{y}$.
\end{prop}

\begin{proof}
We first show that $F\rlser{y}$ is Henselian Write $v=v_t$ for the $t$-adic valuation on $F$, and define partial valuations on $F\rpow{y}$ by setting $v_n(\sum_{j\geq0} f_iy^i)=v(f_n)$. Let $P\in F\rpow{y}[X]$ be a polynomial and $x_0\in F\rpow{y}$ such that $P(x_0)\equiv 0 \text{ mod }y$ and $P'(x_0) \not\equiv 0\text{ mod }y $, we need to show that there exists some $x\in F\rpow{y}$ such that $x\equiv x_0 \text{ mod }y$ and $P(x)=0$. After replacing $P(X)$ by $P(X+x_0)$ we may assume that $x_0=0$.

Write $P = a_mX^m+\ldots+a_0X_0$ with $a_k=\sum_{i\geq0} a_{ki}y^i$, $a_{ki}\in F$ and choose $c,d\in\Z_{\geq0}$ such that $v(a_{ki})\geq -ci-d$ for all $i,k$. Actually, by multiplying $P$ through by a sufficiently hight power of $t$, we may assume that $d=0$, and by increasing $c$ we may also assume that $v_0(P'(x_0))\leq c$. We are going to inductively construct $x=\sum_{i=1}^\infty x_iy^i$ such that:
\begin{enumerate} \item $P(\sum_{i=1}^nx_iy^i)\equiv 0 \text{ mod }y^{n+1}$;
\item $v(x_i) \geq -3ci+c$;
\item $v_0(P'(\sum_{i=1}^nx_iy^i)) = v_0(P'(0))$.
\end{enumerate}
This clearly suffices to prove the claim. So suppose that $x_1,\ldots,x_{n-1}$ have been constructed (note that the same argument with $n=1$ allows us to construct $x_1$ to start the induction).  By Taylor's formula we have
\begin{align*} P(x_1y+\ldots+x_{n}y^{n}) &\equiv P(x_1y+\ldots+x_{n-1}y^{n-1})+x_ny^nP'(x_1y+\ldots x_{n-1}y^{n-1}) \text{ mod }y^{n+1}  \\
&\equiv (\alpha +x_n\beta) y^n \text{ mod }y^{n+1}
\end{align*}
where $\alpha,\beta\in F$ are such that $v(\alpha) = v_n(P(x_1y+x_2y^2+\ldots+x_ny^{n-1}))$ and $v(\beta)=v_0(P'(x_1y+x_2y^2+\ldots+x_{n-1}y^{n-1}))=v_0(P'(x_0))$.  Hence to ensure that $$ P(x_1y+\ldots+x_{n}y^{n}) \equiv 0 \text{ mod }y^{n+1}
$$
we must have $x_n= -\alpha \beta^{-1}$, which is a well defined element of $F$. To see that $v(x_n)\geq -2cm$, note that 
$$ v(x_n) = v_n(P(x_1y+\ldots+x_{n_1}y^{n-1})) - v_0(P'(0)).
$$
and that we can write the coefficient of $y^n$ in the expansion of $P(x_1y+\ldots+x_{n_1}y^{n-1})$ as 
$$ \sum_{i=0}^n \sum_{k=0}^m a_{ki}  \sum_{\substack{i_1+\ldots+i_{n-1}=k \\ i_1+2i_2+\ldots (n-1)i_{n-1}=n-i }} \binom {i} {i_1\ldots i_{n-1}} x_1^{i_1}\ldots x_{n-1}^{i_{n-1}}.
$$
Now, each summand in this has valuation at least as large as
\begin{align*} v(a_{ki}) + i_1v(x_1) + \ldots + i_{n-1}v(x_{n-1}) &\geq - ci - 2ci_1 - 5ci_2 - \ldots - (3(n-1)-c)ci_{n-1} \\ 
&\geq  -ci -2c(n-i) - ci_2-2ci_3 -\ldots - (n-2)ci_{n-1}  \\
&\geq -2cn -c(n-i-k) \geq -3cn+
\end{align*}
and hence using the ultra metric inequality, to show that 
$$ v(x_m) \geq -3cn +c .
$$
it suffices to show that $i+k\geq1$. But if $k=0$ then $i_1=\ldots=i_{m-1}=0$ and hence $i=n$, so we are done. Finally, that $v_0(P'(\sum_{i=1}^nx_iy^i)) = v_0(P'(0))$ is clear, hence $F\rpow{y}$ is Henselian as claimed.

Next let us show that we have an inclusion $F\rlser{u}\subset F\rlser{y}[u]$, it suffices to show that we have an inclusion $F\rpow{u}\subset F\rpow{y}[u]$. Write $P=X^m-a_{m-1}X^{m-1}-\ldots-a_0$, so that we have
$$ u^m = a_{m-1}u^{m-1}+\ldots+a_0
$$ with $a_k\in F\rpow{y}$, say $a_k=\sum_{i\geq1} a_{ki}y^i$ with $a_{01}\neq0$. Now, for any $\sum_{i\geq0} b_iu^i\in F\rpow{u}$ we can repeatedly substitute in $u^m$ for lower powers to give
$$ \sum_{i\geq0} b_iu^i = \sum_{j=0}^{m-1} g_ju^j
$$
with $g_j\in F\pow{y}$. The problem is to show that $g_j\in F\rpow{y}$.

Choose integers $c,d$ such that $-v(b_i)\leq ci+d$, and $-v(a_{ki})\leq ci+d$ for all $i,k$. The point is that after repeatedly substituting in $u^m$ for lower powers we can write
$$ g_j = b_j + \sum_{l\geq m} b_l f_l
$$
where for $l$ in the range $nm+j\leq l < (n+1)m+j$, the term $f_l$ is a sum of products of at least $n$ of the $a_k$. Since each $a_k$ is divisible by $y$, this means that if when calculating $-v_n(g_j)$ we only need to take account of the terms 
$$ b_j + \sum_{l=m}^{(n+1)m+j}b_lf_l
$$
and moreover, only need to take account of the terms for which $f_l$ is a sum of multiples of at most $n$ of the $a_k$. But since we have $-v(a_{ki})\leq ci+d$, when we multiply together at most $n$ of the $a_k$, we get something which satisfies $-v_n\leq (c+d)n $, for all $n\geq1$, in other words $-v_n(f_l)\leq (c+d)n$ for $n\geq 1$. Since $-v(b_l)\leq cl+d$, it follows that  $$-v_n(g_j)\leq (c+d)n+c((n+1)m+j)+d=(c(m+1)+d)n+(d+(j+1)c)$$ for $n\geq1$, hence each $g_j$ is in $F\rpow{y}$ as required. 

Finally, let us show that we have $F\rlser{y}[u]\subset F\rlser{u}$, since $F\rlser{u}$ is a ring containing $u$ it suffices to show that $F\rlser{y}\subset F\rlser{u}$. We first claim that if $g=\sum_{j\geq1} g_ju^j\in F\rpow{u}$ and $f_i\in F$ are such that $-v(f_i)\leq ci+d$ for some $c,d$, then $\sum_if_ig^i$, a priori in $F\lser{u}$, is actually in $F\rlser{u}$, note that we may also assume that $-v(g_j)\leq cj+d$. The point is that when calculating $-v_n(\sum_i f_ig^i)$, we only need to take account of the terms $\sum_{i=0}^n f_ig^i$, since $u$ divides $g$. Once can easily see that since $-v(g_j)\leq cj+d$, we must therefore have $-v_n(g^i) \leq cn+(i+1)d$, and hence $-v_n(\sum_{i=0}^n  f_ig^i)\leq (2c+d)n+2d$.

Applying this claim with $g=y$ we see that to show $F\rlser{y}\subset F\rlser{u}$ it suffices to show that $y\in F\rlser{u}$. Write $y=g(u)=u^m\sum_{j\geq0}g_ju^j$, we now assume that $P$ has the form 
$$ X^m +a_{m-1}X^{m-1}+\ldots+a_0 
$$
which $a_k=\sum_{i\geq1} a_{ki}y^i$ with $-v(a_{ki})$ growing linearly in $i$. We will show that if $-v(g_j)$ grows faster than linearly, then there cannot be the requisite cancellation in
$$ u^m +a_{m-1}u^{m-1}+\ldots+a_0 
$$
to ensure that it equals zero.

If we substitute the expression for $y$ into the above equation, then the term $a_ku^k$, for $k<n$, looks like
$$ \sum_{N=0}^\infty \left(\sum_{i=1}^{\lfloor \frac{N-k}{m} \rfloor} a_{ki}\left(\sum_{j_1+\ldots+j_i=N-in-k}g_{j_1}\ldots g_{j_i}\right)\right) u^N
$$
and there exist arbitrarily large $N$ such that for each $k$, the dominant term in this sum for the coefficient of $u^N$ is the term $a_{k1}g_{N-k-n}$ corresponding to $i=1$, since otherwise this would contradict the faster than linear growth of the $-v(g_j)$ (here we are using the fact that the $-v(a_{ki})$ grow linearly in $i$). Hence again by the faster than linear growth of the $-v(g_j)$, there exist arbitrarily large $N$ for which one of these terms $a_{k1}g_{N-k-n}$ has strictly larger negative valuation than the others. Hence it cannot possibly happen that  
$$ u^m +a_{m-1}u^{m-1}+\ldots+a_0=0
$$
and we obtain our contradiction.
\end{proof}

Our next key result result will be that $\cur{R}_{\ekd}^\mathrm{int}$ itself is Henselian. Recall that by definition, we can write 
$$ \cur{R}_{\ekd}= \bigcup_{s>0} \left( \bigcap_{0<s'\leq s} \left( \bigcup_{\eta<1} A_{\eta,s'}\right) \right)
$$
where $A_{\eta,s'}$ consists of series $\sum_if_iy^i$ with $f_i\in\cur{E}_\eta$ and $\Norm{f_i}_\eta r^{-is'}\rightarrow 0 $. We can define a norm on $A_{\eta,s'}$ given by $$\Norm{\sum_if_iy^i}_{\eta,s'}=\sup_i\left\{\Norm{f_i}_\eta r^{-is'} \right\}$$ which induces a topology. We put the direct limit topology on $\bigcup_{\eta<1}A_{\eta,s'}$, the inverse limit topology on $( \bigcap_{0<s'\leq s} ( \bigcup_{\eta<1} A_{\eta,s'}) )$ and then the direct limit topology on $\cur{R}_{\ekd}$. With respect to this topology, a sequence $f_n=\sum_if_{in}y^i\in \cur{R}_{\ekd}$ tends towards $0$ if and only if for all sufficiently small $s>0$, there exists some $\eta<1$, such that the norms $\Norm{f_n}_{\eta,s}=\sup_i \left\{\Norm{f_{in}}r^{-is}\right\}$ are all defined, and $\Norm{f_n}_{\eta,s}\rightarrow 0$. We first need a lemma.

\begin{lem} \label{etas} \begin{enumerate}\item The norm $\Norm{\cdot}_{\eta,s}$ makes each $A_{\eta,s}$ into a Banach $K$-algebra. 
\item Suppose that $a\in \cur{R}^\mathrm{int}_{\ekd}$ satisfies $\Norm{a}<1$, $\Norm{\cdot}$ denoting the $\pi$-adic norm on $\cur{R}^\mathrm{int}_{\ekd}$. Then there exists some $s$ such that for all $0<s'\leq s$, there exists some $\eta$ such that $a\in A_{\eta,s'}$ and $\Norm{a}_{\eta,s'}<1$.
\end{enumerate}
\end{lem}

\begin{proof} \begin{enumerate}
\item This is entirely standard.
 
\item If $a=\sum_ia_iy^i\in\cur{R}_{\ekd}^\mathrm{int}$ and $\Norm{a}<1$, it follows from discreteness of the $\pi$-adic norm on $\widehat{\cur{R}}_{\ek}^\mathrm{int}$ (which is just another copy of $\cur{O}_{\ek}$ but with `ground field' $\ek$ rather than $K$) that there exists some $c<1$ such that $\Norm{a_i} < c$ for all $i$. Choose $s_0$ such that $a\in \cur{R}_{\ekd,s_0}$, so there exists some $\eta_0$ such that $a_i\in \cur{E}_\eta$ and $\Norm{a_i}_{\eta_0} r^{-is_0}\rightarrow 0$ as $i\rightarrow \pm\infty$, note that for all $i>0$ we have $\Norm{a_i}_{\eta_0} r^{-is_0}<\Norm{a_i}\leq c$. Choose $i_0\leq0$ such that $i \leq i_0\Rightarrow \Norm{a_i}_{\eta_0} r^{-is_0}\leq c$, by increasing $\eta_0$ if necessary we may also assume that $\Norm{a_i}_{\eta_0}< c$ for $i_0 < i \leq 0 $. We can thus choose $s\leq s_0$ to ensure that $\Norm{a_i}_{\eta_0} r^{-is}\leq c$ for all $i_0 < i \leq i_0$, hence $\sup_i \left\{\Norm{a_i}_{\eta_0} r^{-is}\right\}\leq c$, however, we may no longer have that $\Norm{a_i}_{\eta_0} r^{-is}\rightarrow 0$ as $i \rightarrow -\infty$.

But now we know that for all $0<s'\leq s$, $$\sup_i \left\{ \Norm{a_i}_{\eta_0} r^{-is'}\right\}\leq c$$ and that there exists some $\eta$ such that $\Norm{a_i}_\eta r^{-is'}\rightarrow 0$, we may assume that $\eta>\eta_0$, hence $$\sup_i\left\{ \Norm{a_i}_\eta r^{-is'}\right\} \leq c <1$$ as required.

\end{enumerate}
\end{proof}

\begin{prop} \label{hensel} The ring $\cur{R}^\mathrm{int}_{\ekd}$ is a Henselian local ring, with maximal ideal generated by $\pi$.
\end{prop}

\begin{proof} I first claim that $\cur{R}^\mathrm{int}_{\ekd}$ is local, with maximal ideal generated by $\pi$. So consider the embedding $\cur{R}^\mathrm{int}_{\ekd}\rightarrow \widehat{\cur{R}}^\mathrm{int}_{\ek}$ into the $\pi$-adic completion of $\cur{R}_{\ek}^\mathrm{int}$, and suppose that $a\in \cur{R}^\mathrm{int}_{\ekd}$ is such that $a\notin(\pi)$, that is $\Norm{a}=1$ for the natural $\pi$-adic norm on $\cur{R}^\mathrm{int}_{\ekd}$. Since the mod-$\pi$ reduction of $a$ is a unit there exists some $u\in \cur{R}^\mathrm{int}_{\ekd}$ such that $\Norm{au-1}<1$, and if we let $x=1-au$ then the series $\sum_{n\geq0}x^n$ converges in $\widehat{\cur{R}}^\mathrm{int}_{\ek}$ to an inverse $v$ for $au$. By Lemma \ref{etas}(ii) there exists some $s_0$ such that for all $0<s\leq s_0$, there exists some $\eta<1$ such that $a,u\in A_{\eta,s}$ and $\Norm{x}_{\eta,s}<1$. Hence by Lemma \ref{etas}(i) the series $\sum_n x^n$ converges in $A_{\eta,s}$ for all such $s$, and so $v\in \cur{R}_{\ekd}$. Since $\cur{R}_{\ekd}\cap \widehat{\cur{R}}^\mathrm{int}_{\ek}= \cur{R}^\mathrm{int}_{\ekd}$, it follows that $v\in \cur{R}^\mathrm{int}_{\ekd}$, and hence $a$ is a unit in $\cur{R}^\mathrm{int}_{\ekd}$. 

Exactly as in Lemma 3.9 of \cite{padicmono}, to show that $\cur{R}^\mathrm{int}_{\ekd}$ is Henselian, it suffices to show that if $P(x)=x^m-x^{m-1}+a_2x^{m-2}+\ldots+a_m$ is a polynomial with $a_k\in \pi\cur{R}^\mathrm{int}_{\ekd}$ for all $k$, then $P$ has a root $y$ in $\cur{R}^\mathrm{int}_{\ekd}$ such that $y\equiv 1 \;(\mathrm{mod}\;\pi)$. By Hensel's lemma the sequence $y_n$ defined by 
$$y_0=1,\;\; y_{n+1}= y_i - \frac{P(y_n)}{P'(y_n)}$$
converges to such a $y\in \widehat{\cur{R}}^\mathrm{int}_{\ek}$, since $\cur{R}_{\ekd}\cap \widehat{\cur{R}}^\mathrm{int}_{\ek}= \cur{R}^\mathrm{int}_{\ekd}$ we must show that in fact $y_n\rightarrow y$ inside $\cur{R}_{\ekd}$. The proof is almost identical to the usual proof of Hensel's Lemma. Since $\Norm{a_k},\Norm{P(1)P'(1)^{-1}}<1$, by Lemma \ref{etas} we may choose $s$, and for all $0<s'\leq s$ some $\eta=\eta(s')$ such that all the $a_k$ and $P'(1)^{-1}$ are in $A_{\eta,s}$ and $\Norm{a_k}_{\eta,s'},\Norm{P(1)P'(1)^{-1}}_{\eta,s'}<1$. We claim by induction that:
\begin{enumerate}
\item $y_j\in A_{\eta,s'}$, $\Norm{y_{n-1}}_{\eta,s'}\leq1$ and $\Norm{y_n-1}_{\eta,s'}\leq c $,
\item $P'(y_n)$ is invertible in $A_{\eta,s'}$,
\item $\Norm{P(y_n)P'(y_n)^{-1}}_{\eta,s'}\leq c^{2^n}$.
\end{enumerate}

Note that by assumption these are all true for $n=0$. Thus assume that i), ii) and iii) are true for $y_n$. Then we have 
\begin{align*}\Norm{y_{n+1}-1}_{\eta,s'} &\leq \max\left\{\Norm{y_{n+1}-y_n}_{\eta,s'},\Norm{y_{n}-1}_{\eta,s'}\right\} \\
&\leq \max\left\{c^{2^n},c\right\} \leq c \\
\Norm{y_{j+n}}_{\eta,s'} &\leq \max\left\{ \Norm{y_{n+1}-y_n}_{\eta,s'},\Norm{y_n}_{\eta,s'} \right\} \\
&\leq \max\left\{c^{2^n},1\right\} \leq 1
\end{align*}
and thus i) holds for $y_{n+1}$. Hence we can write $P'(y_{n+1}) = 1+ x$ for some $x\in A_{\eta,s'}$ with $\Norm{x}_{\eta,s'} <1$, and thus ii) is also true for $y_{n+1}$. Also note that this implies that $\Norm{P'(y_n)^{-1}}_{\eta,s'}\leq 1$ and hence to prove iii) it suffices to show that $\Norm{P(y_{n+1})}_{\eta,s'}\leq c^{2^{n+1}}$. But now using the Taylor expansion and the fact that $\Norm{y_n}_{\eta,s'}\leq1$ gives
$$ P(y_{n+1}) = P(y_n) - P'(y_n)P(y_n)P'(y_n)^{-1} + z(P(y_n)P'(y_n)^{-1})^2
$$
for some $z\in A_{\eta,s'}$ with $\Norm{z}_{\eta,s'}\leq 1$. Hence 
$$ \Norm{P(y_{n+1})}_{\eta,s} \leq \Norm{P(y_n)P'(y_n)^{-1}}_{\eta,s'}^2 \leq c^{2^{n+1}}
$$
and iii) holds for $y_{n+1}$. Hence $y_n\rightarrow y$ in each $A_{\eta,s'}$, and thus $y\in \cur{R}^\mathrm{int}_{\ekd}$ as required.
\end{proof}

\begin{lem} \label{sameform} Write $F=k\lser{t}$ and suppose that $F\lser{u}/F\lser{y}$ is a separable, finite, totally ramified extension, coming from some $F\rlser{u}/F\rlser{y}$. Let $\cur{R}/\cur{R}_{\ekd}^\mathrm{int}$ be the corresponding unramified finite extension given by Proposition \ref{hensel}. Then there is an isomorphism
$$ \cur{R}_{\ekd}^u \cong \cur{R}\otimes_{\cur{R}_{\ekd}^\mathrm{int}} \cur{R}_{\ekd}  
$$
where $\cur{R}_{\ekd}^u$ is a copy of $\cur{R}_{\ekd}$ but with series parameter $u$.
\end{lem}

\begin{proof} We closely follow the proof of Proposition 3.4 of \cite{matsuda}. Write $y=\bar{g}(u)$, so that $\bar{g}=u^m(\bar{g}_0+\bar{g}_1u+\ldots)$ for $\bar{g}_i\in F$, $\bar{g}_0\neq 0$ and $-v(\bar{g}_i)\leq cj+d$ for some integers $c,d$, $v$ being the $t$-adic valuation on $F$. Lift $\bar{g}$ to some $g\in \cur{R}_{\ekd}^{\mathrm{int},+,u}$ (i.e. a copy of $\cur{R}_{\ekd}^{\mathrm{int},+}$ but with series parameter $u$) of the form $g=u^n(g_0+g_1u+\ldots)$ with $g_0\notin \pi\cur{O}_{\ekd}$. Actually, we want to lift slightly more carefully than this - if we write $\bar{g}_i=\sum_{j\geq -ci-d} \bar{g}_{ij}t^j$ then we will set $ g_i = \sum_{j\geq -c-d} g_{ij}t^j$ where $g_{ij}\in K$ is an element of norm 1 lifting $\bar{g}_{ij}$, and zero if $\bar{g}_{ij}$ is.

Thus $g$ is invertible in $\cur{R}_{\ekd}^{\mathrm{int},u}$, and in fact I claim that for all $\sum_if_iy^i\in \cur{R}_{\ekd}^\mathrm{int}$, the sum $\sum_i f_i g^i$ converges in $\cur{R}_{\ekd}^{\mathrm{int},u}$. Let us treat the sums $\sum_{i\geq0} f_ig^i$ and $\sum_{i<0}f_ig^i$ separately, first let's look at the positive part. Since we have $\Norm{g_i}_\eta \leq \eta^{-ci-d}$ it follows that 
$$ \Norm{g}_{\eta,s} \leq \eta^{-d}r^{-ns} 
$$
at least for $\eta$ and $s$ such that this norm is defined. We therefore have
$$ \Norm{f_ig^i}_{\eta,s} \leq \Norm{f_i}_\eta (\eta^{-d}r^{-ns})^i
$$
and hence, if we are given $s$ and $\eta$ such that $f_i\in \cur{E}_\eta$, $g\in A_{\eta,s}$ and $\Norm{f_i}r^{-is}\rightarrow 0$ as $i\rightarrow\infty$, then since $n>1$, by increasing $\eta$ we can ensure that $\eta^{-d}r^{-ns}\leq r^{-s}$ and hence $\Norm{f_ig^i}_{\eta,s}\rightarrow 0$ as $i\rightarrow\infty$. Hence for all $s$ sufficiently small, there exists some $\eta$ such that $\sum_{i\geq0} f_ig^i$ converges in $A_{\eta,s}$ and thus the sum converges in $\cur{R}_{\ekd}$. It is easy to see that the limit has to have integral coefficients, and hence the sum converges in $\cur{R}_{\ekd}^\mathrm{int}$.

Next let us look at the negative part $\sum_{i<0}f_ig^i$, let us rewrite this as $\sum_{i>0}f_ig^{-i}$ where for all sufficiently small $s$ there exists an $\eta$ with $f_i\in \cur{E}_\eta$ and $\Norm{f_i}_\eta r^{is}\rightarrow 0$ as $i\rightarrow \infty$. Write $g^{-1}=a_0^{-1}u^{-n}(1+b_1u+\ldots)$, then as in the proof of Lemma \ref{quotfieldlin}, we have that each $b_i$ is a sum of things of the form $g_{i_1}\ldots g_{i_m}$ with $i_j\geq1$ and $i_1+\ldots +i_m=i$. Hence we have $\Norm{b_i}_\eta\leq \eta^{-i(c+d)}$, and hence, where defined, we must have $\Norm{g^{-1}}_{\eta,s} \leq \Norm{a_0^{-1}}r^{ns}$. Now choose $s_0>0$ such that for all $0< s \leq s_0$ we have some $\eta$ such that $\Norm{f_i}_\eta r^{is}\rightarrow 0$, and let $s_1=s_0/(n+1)$. Then for any $0<s\leq s_1$ we have
$$  \Norm{f_ig^{-i}}_{\eta,s} \leq \Norm{f_i}_\eta\left(\Norm{a^{-1}_0}_\eta r^{ns}\right)^i
$$
and since $\Norm{a_0^{-1}}=1$, by increasing $\eta$ we can ensure that $\Norm{a_0^{-1}}_\eta r^{ns}\leq r^{s'}$ for some $s'\leq s_0$, and thus by further increasing $\eta$ we can ensure that $\Norm{f_ig^{-i}}_{\eta,s}\rightarrow 0$ and hence the sum converges in $\cur{R}_{\ekd}^u$. Again, it is not hard to see that it must actually lie in $\cur{R}_{\ekd}^{\mathrm{int},u}$.

Hence we get a ring homomorphism $\cur{R}_{\ekd}^\mathrm{int}\rightarrow \cur{R}_{\ekd}^{\mathrm{int},u}$ by sending $\sum_if_iy^i$ to $\sum_if_ig^i$, it is clear that modulo $\pi$ this induces the given map $F\rlser{y}\rightarrow F\rlser{u}$. Note also that there are uniquely determined power series $c_k=\sum_{i\geq1} c_{ki}y^i\in \cur{O}_{\ekd}\pow{y}$ such that 
$$ u^m + c_{m-1}u^{m-1}+\ldots + c_0=0
$$
inside $\cur{O}_{\ekd}\pow{u}$, I claim that in fact these power series lie inside $\cur{R}_{\ekd}^\mathrm{int,+}$, it suffices to show that they lie in $\cur{R}_{\ekd}^+$. In fact, one can show inductively using the equation
$$ u^m +\sum_{N=m}^\infty\left(\sum_{k=0}^{m-1}\sum_{j=1}^{\lfloor \frac{N-k}{m} \rfloor}c_{kj}\sum_{k_1+\ldots+k_j=N-jm-k}g_{k_1}\ldots g_{k_j}\right)u^N=0
$$
determining the $c_{ki}$ that if $\eta,s$ are such that $\Norm{a_i}_\eta\leq Cr^{is}$ for some constant $C$, then we have
$$ \Norm{c_{ki}}\leq (Cr^s)^{n(i-1)+k}\Norm{g_0^{-1}}_\eta^{n(i-1)+k+i}
$$
for all $i,k$. Since by increasing $\eta$ we may make both $C$ and $\Norm{g_0^{-1}}_\eta$ as close to $1$ as we please, it therefore follows that for all $k$, there exists some $s_0$ such that for all $0<s\leq s_0$ and all $\lambda>1$, there exists some $\eta<1$ and $D>0$ such that 
$$  \Norm{c_{ki}}_\eta\leq D (\lambda r^s)^{(n+1)i}
$$
for all $i$. Hence the series $c_k$ are in $\cur{R}_{\ekd}^+$ as required, and $u$ is actually integral over $\cur{R}^\mathrm{int}_{\ekd}$. It then follows that 
$$ \cur{R}^\mathrm{int}_{\ekd}\rightarrow \cur{R}_{\ekd}^{\mathrm{int},u}=\cur{R}^\mathrm{int}_{\ekd}[u]
$$
is a finite, $\pi$-adically unramified extension with induced extension $F\rlser{y}\rightarrow F\rlser{y}$ of residue fields, we must therefore have $\cur{R}'\cong \cur{R}_{\ekd}^{\mathrm{int},u}$, or in other words $\cur{R}_{\ekd}^{\mathrm{int},u}$ is an explicit construction of the lift $\cur{R}'$ of $F\rlser{u}$. Finally, we need to prove that we have $\cur{R}_{\ekd}^{\mathrm{int},u} \otimes_{\cur{R}_{\ekd}^\mathrm{int}} \cur{R}_{\ekd} \cong \cur{R}_{\ekd}^u$.  But the exact same argument as in the integral case shows that for any series $\sum_if_iy^i$ in $\cur{R}_{\ekd}$, the series $\sum_if_ig^i$ converges in $\cur{R}_{\ekd}^u$, and we therefore get a finite map $\cur{R}_{\ekd}\rightarrow \cur{R}_{\ekd}^u$ such that $\cur{R}_{\ekd}^u=\cur{R}_{\ekd}[u]$. We therefore get a commutative push-out diagram
$$ \xymatrix{ \cur{R}^{\mathrm{int},u}_{\ekd} \ar[r] & \cur{R}^u_{\ekd} \\
\cur{R}^\mathrm{int}_{\ekd} \ar[r]\ar[u] & \cur{R}_{\ekd} \ar[u] }
$$
which realises $\cur{R}_{\ekd}^u$ as the tensor product $\cur{R}_{\ekd}^{\mathrm{int},u} \otimes_{\cur{R}_{\ekd}^\mathrm{int}} \cur{R}_{\ekd}$. 
\end{proof}

Having established the required properties of $\cur{R}_{\ekd}$, we can now introduce the key objects of study in this section, namely $(\varphi,\nabla)$-modules over $\cur{R}_{\ekd}$. 

\begin{defn} A Frobenius on $\cur{R}_{\ekd}$ is a continuous ring endomorphism, $\sigma$-linear over $\ekd$, lifting the absolute $q$-power Frobenius on $k\lser{t}\rlser{y}$.
\end{defn}

Fix a Frobenius $\sigma$ on $\cur{R}_{\ekd}$, and let $\partial_y:\cur{R}_{\ekd}\rightarrow \cur{R}_{\ekd}$ be the derivation given by differentiation with respect to $y$, that is $\partial_y(\sum_if_iy^i)=\sum_i if_i y^{i-1}$.
\begin{defn} \begin{itemize}
\item A $\varphi,$-module over $\cur{R}_{\ekd}$ is a finite free $\cur{R}_{\ekd}$-module $M$ together with a Frobenius structure, that is an $\sigma$-linear map $$\varphi: M \rightarrow M$$ which induces an isomorphism $M\otimes_{\cur{R}_{\ekd},\sigma} \cur{R}_{\ekd}\cong M$.
\item A $\nabla$-module over $\cur{R}_{\ekd}$ is a finite free $\cur{R}_{\ekd}$-module $M$ together with a connection, that is an $\ekd$-linear map $$\nabla:M\rightarrow  M$$ such that $\nabla(fm)=\partial_y(f)m+f\nabla(m)$ for all $f\in\cur{R}_{\ekd}$ and $m\in M$.
\item A $(\varphi,\nabla)$-module over $\cur{R}_{\ekd}$ is a finite free $\cur{R}_{\ekd}$-module $M$ together with a Frobenius $\varphi$ and a connection $\nabla$, such that the diagram
$$
\xymatrix{
M \ar[r]^\nabla \ar[d]^{\varphi} & M \ar[d]^{\partial_y(\sigma(y))\varphi} \\
M \ar[r]^\nabla & M
}
$$
commutes.
\end{itemize}
If $M$ is a $\nabla$-module over $\cur{R}_{\ekd}$ we define its cohomology to be
\begin{align*} H^0(M) &:= \ker( \nabla )\\
H^1(M) &:= \mathrm{coker}( \nabla).
\end{align*}
\end{defn}
It will also be useful to interpret these in a more co-ordinate free fashion, to do so let $\Omega^1_{\cur{R}_{\ekd}}$ be the free $\cur{R}_{\ekd}$-module generated by $dy$, and
$$d:\cur{R}_{\ekd}\rightarrow \Omega^1_{\cur{R}_{\ekd}}$$
given by $df=\partial_y(f)dy$. Then $(\Omega^1_{\cur{R}_{\ekd}},d)$ is universal for continuous $\ekd$-derivations from $\cur{R}_{\ekd}$ into separated topological $\cur{R}_{\ekd}$-modules, and a connection on an $\cur{R}_{\ekd}$-module $M$ is equivalent to a homomorphism
$$ \nabla:M \rightarrow M\otimes \Omega^1_{\cur{R}_{\ekd}}
$$
such that $\nabla(fm)=m\otimes df + f\nabla(m)$. 

\begin{defn} \label{unip}A $\nabla$-module $M$ over $\cur{R}_{\ekd}$ is said to be unipotent if there exists a basis $\left\{e_1,\ldots,e_n\right\}$ of $M$ such that 
$$
\nabla(e_i)\in \cur{R}_{\ekd}e_1 + \ldots + \cur{R}_{\ekd}e_{i-1}
$$
for all $i$. We say that a $(\varphi,\nabla)$-module is unipotent if the underlying $\nabla$-module is.
\end{defn}

We will be using Theorem 6.1.2 of \cite{padicmono} as a template for the theorem we wish to prove, so we will need to be able to associate a finite extension of $\cur{R}_{\ekd}$ to certain kinds of `nearly finite separable' extensions of $k\lser{t}\lser{y}$, that is composite extensions of the form
$$ k\lser{t}\rlser{y} \rightarrow k\lser{t}^{1/p^m}\rlser{y} \rightarrow F \rlser{y} \rightarrow F\rlser{u}
$$
where $F/k\lser{t}^{1/p^m}$ is a finite separable extension and $F\rlser{u}/F\rlser{y}$ is finite, Galois and totally ramified. We will consider each of these extensions in turn, starting with the extension $k\lser{t}\rlser{y} \rightarrow k\lser{t}^{1/p^m}\rlser{y}$. If $\sigma$ is our Frobenius on $\ekd$, then there is an induced map $\cur{R}_{\ekd}\overset{\sigma'}{\rightarrow} \cur{R}_{\ekd}$ given by $\sum_i f_iy^i \mapsto \sum_i \sigma(f_i)y^i$, note that this should not be confused with a Frobenius on $\cur{R}_{\ekd}$. We let $\cur{R}_{(\ekd)^{\sigma^{-m}}}$ denote $\cur{R}_{\ekd}$ considered as an $\cur{R}_{\ekd}$-algebra via the $m$-fold composition of $\sigma'$, so that $\cur{R}_{\ekd}\rightarrow \cur{R}_{(\ekd)^{\sigma^{-m}}}$ is a lifting of $k\lser{t}\rlser{y} \rightarrow k\lser{t}^{1/p^m}\rlser{y}$.

Secondly, if $F/k\lser{t}$ is a finite separable extension, then we can consider the finite extension $\cur{E}_K^{\dagger,F}/\ekd$ as in \S5  of \cite{rclsf1}. Since $\cur{E}_K^{\dagger,F}$ is of the same form as $\ekd$ (but with a different parameter and ground field) we may define $\cur{R}_{\cur{E}_K^{\dagger,F}}$ exactly as above, and there is a natural map
$$ \cur{R}_{\cur{E}_K^{\dagger}}\rightarrow \cur{R}_{\cur{E}_K^{\dagger,F}}.
$$
Actually, these are both particular cases of a more general construction associated to a finite extension $\cur{F}^\dagger/\ekd$. The point is that we can use the fringe topology on $\ekd$ to induce a similar fringe topology on $\cur{F}^\dagger$, by writing
$$
\cur{F}^\dagger=\mathrm{colim}_\eta  \cur{F}_\eta
$$
where each $\cur{F}_{\eta}$ is a finite free $\cur{E}_\eta$-module (namely the sub-$\cur{E}_\eta$-module of $\cur{F}^\dagger$ spanned by some chosen basis for $\cur{F}^\dagger/\ekd$). These then come with a compatible collection of topologies induced by some Banach norm on each $\cur{F}_\eta$, and we can give $\cur{F}^\dagger$ the direct limit topology. This does not depend on the choice of basis. Thus we can define the Robba ring $\cur{R}_{\cur{F}^\dagger}$ over $\cur{F}^\dagger$ exactly as in Definition \ref{robba}, this does not depend on the choice of $\cur{F}_\eta$ or their Banach norms, and we can topologise it exactly as we topologise the Robba ring $\cur{R}_{\ekd}$. We can also describe $\cur{R}_{\cur{F}^\dagger}$ more straightforwardly as follows.

\begin{lem} \label{finiterobbabase} Let $\cur{F}^\dagger/\ekd$ be a finite extension. Then the natural multiplication map $$\cur{R}_{\ekd}\otimes_{\ekd} \cur{F}^\dagger\rightarrow  \cur{R}_{\cur{F}^\dagger}$$ is an isomorphism. Hence in particular, $\cur{R}_{\ekd}\otimes_{\ekd} \cur{E}^{\dagger,F}_K \cong \cur{R}_{\cur{E}_K^{\dagger,F}}$
\end{lem}

\begin{proof} Both are subrings of the ring $\cur{F}^\dagger\pow{y^{-1},y}$ of doubly infinite series with coefficients in $\cur{F}^\dagger$, hence the map is injective. To prove surjectivity, let $v_1,\ldots,v_n$ be a basis for $\cur{F}^\dagger/\ekd$. Then for
$$ f = \sum_{i} f_i y^i \in \cur{R}_{\cur{F}^\dagger}$$ we know that we can write
$$ f = \sum_j (  \sum_{i}f_{ij}y^i )v_j$$
where $f_i=\sum_{j} f_{ij}v_j\in \cur{F}^\dagger$, we must show that each $\sum_{i}f_{ij}y^i $ actually lies in $\cur{R}_{\ekd}$. So let $s>0$ be sufficiently small. Then there exists some $\eta<1$ such that each $f_i\in\cur{F}_\eta$ and $\Norm{f_i}_{\eta}r^{-is}\rightarrow 0$. But $f_i\in\cur{F}_\eta$ implies that $f_{ij}\in \cur{E}_\eta$ for all $i,j$, and since $\cur{F}_\eta$ is a finite free module over the Banach algebra $\cur{E}_\eta$, $\Norm{f_i}_{\eta}r^{-is}\rightarrow 0$ if and only if $\Norm{f_{ij}}_{\eta}r^{-is}\rightarrow0$ for all $j$.  
\end{proof}

Finally we consider the extension $F\rlser{y}\rightarrow F\rlser{u}$. By Proposition \ref{hensel} we can lift this uniquely to some finite, \'{e}tale, Galois extension $\cur{R}^\mathrm{int}_{\cur{E}_K^{\dagger,F}}\rightarrow \cur{R}$. Define $\cur{R}'=\cur{R}_{\cur{E}_K^{\dagger,F}}\otimes_{\cur{R}^\mathrm{int}_{\cur{E}_K^{\dagger,F}}} \cur{R}$, this is isomorphic to $\cur{R}_{\cur{E}_K^{\dagger,F}}^u$ by Lemma \ref{sameform}. Thus associated to the series of extensions 
$$ k\lser{t}\rlser{y} \rightarrow k\lser{t}^{1/p^m}\rlser{y} \rightarrow F \rlser{y} \rightarrow F\rlser{u}
$$
we get extensions
$$ \cur{R}_{\ekd} \rightarrow \cur{R}_{(\ekd)^{\sigma^{-m}}} \rightarrow \cur{R}_{\cur{E}_K^{\dagger,F}}\rightarrow \cur{R}^u_{\cur{E}_K^{\dagger,F}}.
$$
We next explain how to base extend a $\nabla$-module $M$ over $\cur{R}_{\ekd}$ up this tower of extensions. For the extensions $$\cur{R}_{\ekd}\rightarrow \cur{R}_{(\ekd)^{\sigma^{-m}}} \rightarrow \cur{R}_{\cur{E}_K^{\dagger,F}}$$ this is straightforward, we just use the fact that $\cur{R}_{\cur{E}_K^{\dagger,F}} = \cur{R}_{\ekd} \otimes_{\ekd} \cur{E}_K^{\dagger,F}$ and extend the connection $M\rightarrow M$ linearly over $\cur{E}_K^{\dagger,F}$ in the obvious fashion. For the extension $\cur{R}_{\cur{E}_K^{\dagger,F}}\rightarrow \cur{R}_{\cur{E}_K^{\dagger,F}}^u$, we note that the appropriate universal properties of $\Omega^1$ (and the fact that $\cur{R}_{\cur{E}_K^{\dagger,F}}^u$ is a separated topological $\cur{R}_{\cur{E}_{K}^{\dagger,F}}$-module) give a commutative diagram 
$$\xymatrix{ \cur{R}_{\cur{E}_K^{\dagger,F}} \ar[r]\ar[d] & \Omega^1_{\cur{R}_{\cur{E}_K^{\dagger,F}}} \ar[d] \\
\cur{R}^u_{\cur{E}_K^{\dagger,F}} \ar[r] & \Omega^1_{\cur{R}^u_{\cur{E}_K^{\dagger,F}}}
}
$$
with $\Omega^1_{\cur{R}_{\cur{E}_K^{\dagger,F}}} \rightarrow \Omega^1_{\cur{R}^u_{\cur{E}_K^{\dagger,F}}}$ linear over $\cur{R}_{\cur{E}_K^{\dagger,F}}$. Hence we can extend
$$ \nabla : M \rightarrow M\otimes \Omega^1_{\cur{R}_{\cur{E}_K^{\dagger,F}}}
$$ to a morphism
$$ \nabla': M\otimes \cur{R}_{\cur{E}_K^{\dagger,F}}^u \rightarrow M\otimes \Omega^1_{\cur{R}_{\cur{E}_K^{\dagger,F}}^u}
$$
by setting $\nabla'(m\otimes r) = \nabla(m)\otimes r + m \otimes dr$. We can now state the version of the $p$-adic monodromy theorem we wish to prove.

\begin{theo} \label{qu} Let $M$ be a $(\varphi,\nabla)$ module over the Robba ring $\cur{R}_{\ekd}$. Then $M$ is quasi-unipotent, that is there exists an integer $m\geq0$, a finite separable extension $F/k\lser{t}^{1/p^m}$ and a finite, Galois, totally ramified extension $F\rlser{u}/F\rlser{y}$ such that $M\otimes \cur{R}^{u}_{\cur{E}_K^{\dagger,F}}$ is unipotent.
\end{theo}

As mentioned above, the proof will closely mirror Kedlaya's proof of Theorem \ref{keduni}, which uses the `usual' $p$-adic monodromy theorem for the completion of the fraction field of $A$, and then `descending' horizontal sections to $\cur{R}_B$ for some localisation $B$ of $A$. Our proof will proceed entirely similarly, with $A$ being replaced by $\ekd$ and the completion of its fraction field by $\ek$, we do not have to worry about the localisation, since $\ekd$ is already a field.

Thus we will deduce Theorem \ref{qu} from the corresponding statement for $\ek$, which we now recall. Associated to the extensions
$$ k\lser{t}\rlser{y} \rightarrow k\lser{t}^{1/p^m}\rlser{y} \rightarrow F \rlser{y} \rightarrow F\rlser{u}$$
we get the extensions
$$ k\lser{t}\lser{y} \rightarrow k\lser{t}^{1/p^m}\lser{y} \rightarrow F \lser{y} \rightarrow F\lser{u}$$
and therefore the extensions
$$ \cur{R}_{\ek} \rightarrow \cur{R}_{(\ek)^{\sigma^{-m}}} \rightarrow \cur{R}_{\cur{E}_K^{F}}\rightarrow  \cur{R}_{\cur{E}_K^F}^u
$$
as in Section 3 of \cite{padicmono}, where, for example, $\cur{E}_K^F$ is the finite extension of $\ek$ corresponding to some $F/k\lser{t}$ (beware that the notations in \cite{padicmono} are very different, there these rings are denoted $\Gamma^{k\lser{t}\lser{y}}_{\mathrm{an,con}},\Gamma^{k\lser{t}^{1/p^m}\lser{y}}_{\mathrm{an,con}},\Gamma^{F\lser{y}}_{\mathrm{an,con}}$ and $\Gamma^{F\lser{u}}_{\mathrm{an,con}}$ respectively). Then over $\ek$ Kedlaya's $p$-adic monodromy theorem is the following.

\begin{theo}[\cite{padicmono}, Theorem 6.1.2] \label{altkeduni}Let $M$ be a $(\varphi,\nabla)$ module over the Robba ring $\cur{R}_{\ek}$. Then there exists an integer $m\geq0$, a finite separable extension $F/k\lser{t}^{1/p^m}$, and a finite, Galois, totally ramified extension $F\lser{u}/F\lser{y}$ such that $M\otimes \cur{R}^{u}_{\cur{E}_K^{F}}$ is unipotent. 
\end{theo}

If we choose compatible Frobenii on $\cur{R}_{\cur{E}_K^\dagger}$ and $\cur{R}_{\ek}$ (which we may do) there is a natural base extension functor $M\mapsto M\otimes_{\cur{R}_{\ekd}}\cur{R}_{\ek}$ from $(\varphi,\nabla)$-modules over $\cur{R}_{\ekd}$ to $(\varphi,\nabla)$-modules over $\cur{R}_{\ek}$. Exactly as in \cite{kedlayafiniteness}, the key stage in the proof of Theorem \ref{qu} will be to show the following. 

\begin{prop} \label{unibase} Suppose that $M$ is a $\nabla$-module over $\cur{R}_{\ekd}$. If $M\otimes_{\cur{R}_{\ekd}}\cur{R}_{\ek}$ is unipotent (as a $\nabla$-module over $\cur{R}_{\ek})$, then $M$ is unipotent.
\end{prop}

In order to prove Proposition \ref{unibase} we will need to adapt Kedlaya's method of producing horizontal sections to our situation. Happily, this can be achieved entirely straightforwardly. As in \cite{kedlayafiniteness}, we first need to introduce some auxiliary rings. For any $\eta<1$ and $s>0$ we let $R_{\cur{E}_{\eta},s}$ denote the ring of formal series $\sum_i f_iy^i$ such that there exists $c>0$ with 
$$
\Norm{f_i}_{\eta}r^{-is+\norm{i}c}\rightarrow 0
$$ 
as $i\rightarrow \pm\infty$. We let $R_{\ekd,s}$ be the direct limit $\mathrm{colim}_{\eta<1} R_{\cur{E}_\eta,s}$. For $R=R_{\ekd,s}$ or $R=R_{\ek,s}$ there is an obvious notion of $\nabla$-module over $R$, as well as unipotence for such modules. 

Each $R_{\cur{E}_\eta,s}$ has a norm given by 
$$ \Norm{\sum_if_iy^i}_{\eta,s} = \sup_i \left\{\Norm{f_i}_\eta r^{-is}\right\}
$$ and there is a similarly defined norm on $R_{\ek,s}$, all are complete with respect to these norms. The ring $R_{\ekd,s}$ is given the direct limit topology from the topologies on each $R_{\cur{E}_\eta,s}$.

\begin{defn} Let $A$ be either $\ekd$ or $\ek$, and let $M$ be a free $\nabla$-module over either $\cur{R}_{A}$ or $R_{A,s}$ for some $s>0$. Define $D:M\rightarrow M$ by $D(m)=y\nabla(m)$. Then a strongly unipotent basis for $M$ is a basis $e_1,\ldots,e_n$ such that 
$$ D(e_i) \in Ae_1+\ldots+Ae_{i-1}.
$$
\end{defn}

Note that by Proposition 5.2.6 of \cite{kedlayafiniteness}, when $A=\ek$ then any free, unipotent $\nabla$-module over $\cur{R}_A$ or $R_{A,s}$ admits a strongly unipotent basis, and in fact exactly the same proof shows that the same is true when $A=\ekd$. Also, if $e_1,\ldots,e_n$ is a strongly unipotent basis for $m$, then the kernel of $\nabla$ on $M$ is equal to the kernel of $D$ on the $A$-span on the $e_i$.

\begin{lem} \label{epsnorm} Fix $\eta_0<1$. For any rational $\epsilon\in (0,1]$ there exists some $\eta_0\leq \eta<1$ such that
$$ \Norm{f}_\eta \leq  \Norm{f}^{1-\epsilon}\Norm{f}^\epsilon_{\eta_0}
$$
for all $f\in \cur{E}_{\eta_0}$. Here $\Norm{\cdot}_\eta$ the is natural norm on $\cur{E}_\eta$, and $\Norm{\cdot}$ the $\pi$-adic norm on $\ek$.
\end{lem}

\begin{proof} For $f=\sum_i a_i t^i\in\cur{E}_{\eta_0}$ we write $f_+=\sum_{i\geq 0} a_it^i$ and $f_{-}=\sum_{i<0} a_it^i$, I claim that it suffices to find some $\eta$ that works for all $f_{-}$, indeed if so then we would have
\begin{align*} \Norm{f}^{1-\epsilon}\Norm{f}_{\eta_0}^\epsilon &= \max\left\{ \Norm{f_+}^{1-\epsilon},\Norm{f_-}^{1-\epsilon}\right\}\cdot \max\left\{ \Norm{f_+}^\epsilon_{\eta_0},\Norm{f_{-}}^\epsilon_
{\eta_0} \right\} \\
&\geq \max\left\{ \Norm{f_+}^{1-\epsilon}\Norm{f_+}^\epsilon_{\eta_0} , \Norm{f_-}^{1-\epsilon}\Norm{f_-}^\epsilon_{\eta_0} \right\} \\
&\geq \max \left\{ \Norm{f_+}_{\eta}, \Norm{f_-}_{\eta} \right\} \\
&= \Norm{f}_{\eta}
\end{align*}
since $\Norm{f_+}^{1-\epsilon}\Norm{f_+}^\epsilon_{\eta_0}=\Norm{f_+}_{\eta}=\Norm{f_+}$. But now the change of coordinate $z=y^{-1}$ converts the question for $f_{-}$ into Proposition 2.4.2 of \cite{kedlayafiniteness} in the case of the one-dimensional MW-algebra $K\weak{z}$.
\end{proof}

\begin{lem} \label{pamkey} Let $M$ be a free $\nabla$-module over $R_{\ekd,s}$ for some $s>0$ such that $M':=M\otimes_{R_{\ekd,s}}  R_{\ek,s}$ is unipotent. Let $e$ be the nilpotency index of the matrix via which $D$ acts on a strongly unipotent basis of $M'$. Define functions $f_n:M'\rightarrow M'$ by $f_0(m)=D^{e-1}(m)$ and
$$
f_n(m)=\left(1-\frac{D^2}{l^2}\right)^ef_{n-1}(m)
$$
for $n\geq1$. Then for any $m\in M'$, the sequence $f_n(m)$ converges as $n\rightarrow\infty$ to some element $f(m)$ such that $\nabla(f(m))=0$. Moreover, if $m\in M\subset  M'$ then so is $f(m)$.
\end{lem}

\begin{proof} The fact that $f_n(m)$ converges as $n\rightarrow \infty$ to some $f(m)$ such that $\nabla(f(m))=0$ is exactly Lemma 5.3.1 of \cite{kedlayafiniteness}, we must show that if $m\in M$ then so is $f(m)$, or in other words we need to show that the sequence $f_n(m)$ actually converges in $M$. As in lemma 5.3.2 of \emph{loc. cit.}, choose a basis $e_1,\ldots,e_n$ for $M$ and define the matrix $N$ by
$$ De_j=\sum_i N_{ij}e_i.
$$
Also write $m=\sum_i g_{i,-1}e_i$ and choose some $\eta$ such that each $N_{ij}$ and $g_{i,-1}$ are all defined over $R_{\cur{E}_\eta,s}$. Write $f_n(m)=\sum_i g_{in}e_i$ so that each $g_{in}\in R_{\cur{E}_\eta,s}$, we need to prove that each sequence $g_{in}$ converges to some $g_i\in R_{\cur{E}_\eta,s}$. Exactly as in the proof of Lemma 5.3.2 of \emph{loc. cit.} we know that there exist constants $D>0,0<\lambda<1,\rho>1$ such that
\begin{align*} \Norm{g_{i(n+1)}-g_{in}}_s &\leq D\lambda^n\\
\Norm{g_{i(n+1)}-g_{in}}_{\eta,s} &\leq D\rho^n 
\end{align*}
where $\Norm{\cdot}_s$ is the natural norm on $R_{\ek,s}$ and $\Norm{\cdot}_{\eta,s}$ that on $R_{\cur{E}_\eta,s}$. Thus by Lemma \ref{epsnorm} above we can find some $\eta<\eta'<1$ and $D'>0,0<\lambda'<1$ such that 
$$ \Norm{g_{i(n+1)}-g_{in}}_{\eta',s} \leq D'(\lambda')^n
$$
and hence each sequence $g_{in}$ converges in $R_{\cur{E}_{\eta'},s}$.
\end{proof}

\begin{proof}[Proof of Proposition \ref{unibase}] We first claim that if $M$ is a free $\nabla$-module over $R_{\ekd,s}$ for some $s>0$, such that $M\otimes R_{\ek,s}$ is unipotent, then $M$ is unipotent. Let $D$, $e$ and $f$ be as in Lemma \ref{pamkey} above, and let $M_0$ denote the $\ek$-span of a strongly unipotent basis.

Now, for all $m\in M\otimes  R_{\ek,s}$, $\nabla(f(m))=0$ and hence $f(m)$ must lie in the image of $D^{e-1}$ on $M_0$. Since $f(m)=D^{e-1}(m)$ for any $m\in M_0$, it follows that the image of $f$ is exactly the image of $D^{e-1}$ on $M_0$. But since $M\otimes_{\ekd} \ek$ is dense in $M\otimes R_{\ek,s}$ it follows that $f(M\otimes_{\ekd} \ek) = D^{e-1}(M_0)$.  Hence $f(m)\neq0$ for some $m\in M$, and thus exists some non-zero $n =f(m)\in M$ with $\nabla(n)=0$. By Corollary 5.2.5 of \cite{kedlayafiniteness}, the $R_{\ek,s}$-submodule of $M\otimes R_{\ek,s}$ generated by $m$ is a direct summand, since $m$ belongs to the $\ek$-span of a strongly unipotent basis for $M\otimes R_{\ek,s}$, and hence the $\cur{R}_{\ekd,s}$-submodule of $M$ spanned by $m$ is a direct summand. Thus by quotienting out by this submodule and using induction on the rank of $M$ we get the claimed result.

Now suppose that we have some free $\nabla$-module $M$ over $\cur{R}_{\ekd}$ such that $M\otimes \cur{R}_{\ek}$ is unipotent, and let $e_1,\ldots,e_n$ be a basis for $M$. Then as in Proposition 5.4.1 of \cite{kedlayafiniteness}, if we let $N$ denote the $\cur{R}_{\ekd}\cap R_{\ekd,s}$-span on the $e_i$, then since $\cur{R}_{\ek}\subset \bigcup_{s>0}R_{\ek,s}$, for $s$ sufficiently small, $N\otimes(\cur{R}_{\ek}\cap R_{\ek,s})$ is unipotent. Hence by what we have proved above, $N\otimes R_{\ekd,s}$ is unipotent, and thus admits a strongly unipotent basis $\left\{v_i\right\}$. By an argument identical to Proposition 5.2.6 of \cite{kedlayafiniteness}, $N\otimes(\cur{R}_{\ek}\cap R_{\ek,s})$ also admits a strongly unipotent basis $\left\{w_i\right\}$, and by Corollary 5.2.5 of \emph{loc. cit.}, these two bases have the same $\ek$-span inside $N\otimes R_{\ek,s}$. Hence the $v_i$ form a strongly unipotent basis of $N\otimes (\cur{R}_{\ek}\cap R_{\ekd,s})$. Since $ \cur{R}_{\ek}\cap R_{\ekd,s}\subset \cur{R}_{\ekd}$, it thus follows that $M= N\otimes \cur{R}_{\ekd}$ is unipotent, completing the proof.
\end{proof}

With Proposition \ref{unibase} out of the way, we can complete the proof of Theorem \ref{qu}.

\begin{proof}[Proof of Theorem \ref{qu}]  Let $m,F$ and $F\lser{u}/F\lser{y}$ be as in Theorem \ref{altkeduni}, so that we have extensions
$$ \cur{R}_{\ek} \rightarrow \cur{R}_{(\ek)^{\sigma^{-m}}} \rightarrow \cur{R}_{\cur{E}_K^{F}}\rightarrow \cur{R}^{u}_{\cur{E}_K^{F}}
$$
such that $M\otimes \cur{R}^{u}_{\cur{E}_K^{F}}$ is unipotent. We may assume that the extension $F\lser{u}/F\lser{y}$ arises from some extension $F\rlser{u}/F\rlser{y}$. There is therefore a commutative diagram
$$ \xymatrix{  \cur{R}_{\ekd} \ar[r]\ar[d] & \cur{R}_{(\ekd)^{\sigma^{-m}}} \ar[r]\ar[d] & \cur{R}_{\cur{E}_K^{\dagger,F}}\ar[r]\ar[d] &  \cur{R}^{u}_{\cur{E}_K^{\dagger,F}} \ar[d] \\
\cur{R}_{\ek} \ar[r] & \cur{R}_{(\ek)^{\sigma^{-m}}} \ar[r] & \cur{R}_{\cur{E}_K^{F}}\ar[r] & \cur{R}^{u}_{\cur{E}_K^{F}}
}
$$
and hence if we let $N= M\otimes \cur{R}^{u}_{\cur{E}_K^{\dagger,F}}$, then $N\otimes \cur{R}^{u}_{\cur{E}_K^{F}}$ is unipotent. Hence by Proposition \ref{unibase}, $N$ is unipotent. 
\end{proof}

Using the monodromy theorem, we can now prove finite dimensionality and base change for the cohomology of $(\varphi,\nabla)$-modules over $\cur{R}_{\ekd}$. Again, assume that we have chosen compatible Frobenii on $\cur{R}_{\ekd}$ and $\cur{R}_{\ek}$.

\begin{theo} \label{robbabase} Let $M$ be a $(\varphi,\nabla)$-module over $\cur{R}_{\ekd}$. Then the cohomology groups
$$
H^0(M),H^1(M)
$$
are finite dimensional over $\ekd$, and the base change morphisms
\begin{align*} H^0(M)\otimes_{\ekd}\ek&\rightarrow H^0(M\otimes \cur{R}_{\ek})\\
H^1(M)\otimes_{\ekd}\ek&\rightarrow H^1(M\otimes \cur{R}_{\ek})
\end{align*}
as isomorphisms.
\end{theo}

\begin{rem} Of course, $H^0(M\otimes \cur{R}_{\ek})$ and $H^1(M\otimes \cur{R}_{\ek})$ are defined entirely similarly to $H^0(M)$ and $H^1(M)$.
\end{rem}

\begin{proof} 

This is entirely similar to Proposition 7.2.1 of \cite{kedlayafiniteness}, since $H^i(M\otimes \cur{R}_{\ek})$ is finite dimensional, it suffices to prove the base change statement. First suppose that $M$ is unipotent, with unipotent basis $\left\{m_1,\ldots,m_n \right\}$. Let $M_j$ be the span of $m_1,\ldots,m_j$, this is a sub-$\nabla$-module of $M$, and $M_j/M_{j-1}$ is isomorphic, as a $\nabla$-module, to $\cur{R}_{\ekd}$. Hence direct calculation for $M=\cur{R}_{\ekd}$ together with the snake lemma and induction on $j$ imply the result. Of course, exactly the same calculation works over $\cur{R}_{\cur{F}^\dagger}$ for any finite extension $\cur{F}^\dagger/\ekd$.

In general, we choose $m,F$ and $F\rlser{u}/F\rlser{y}$ such that the conclusions of Theorem \ref{qu} hold. By Lemma \ref{finiterobbabase} we know that 
$$
H^i(M) \otimes_{\ekd} \cur{E}_K^{\dagger,F} \cong H^i(M\otimes_{\cur{R}_{\ekd}} \cur{R}_{\cur{E}_K^{\dagger,K}})
$$
and since a similar calculation holds when replacing $\ekd$ and $\cur{E}_K^{\dagger,F}$ by their completions $\ek$ and $\cur{E}_K^F$ respectively, it suffices to prove the theorem after base changing to $\cur{R}_{\cur{E}_K^{\dagger,F}}$. In other words, after replacing $k\lser{t}$ by $F$ we may assume that there exists a finite, Galois, totally ramified extension $k\lser{t}\rlser{u}/k\lser{t}\rlser{y}$ and corresponding extension $\cur{R}/\cur{R}_{\ekd}^\mathrm{int}$ such that $\cur{R}\otimes_{\cur{R}_{\ekd}^\mathrm{int}} \cur{R}_{\ekd}\cong \cur{R}_{\ekd}^u$ and $M\otimes \cur{R}_{\ekd}^{u}$ is unipotent. 

Since $\cur{R}/\cur{R}_{\ekd}^\mathrm{int}$ is Galois, we can define a trace map $\mathrm{tr}:\cur{R}^u_{\ekd}\rightarrow \cur{R}_{\ekd}$ by summing over $\mathrm{Aut}(\cur{R}/\cur{R}_{\ekd}^\mathrm{int})$, we similarly get a trace map $\mathrm{tr}: \Omega^1_{\cur{R}^u_{\ekd}}\rightarrow \Omega^1_{\ekd}$ such that the diagram
$$ \xymatrix{ M \ar[r] \ar[d]^\nabla & M \otimes \cur{R}_{\ekd}^u \ar[r]^{\mathrm{tr}}\ar[d]^\nabla & M \ar[d]^\nabla \\
M\otimes \Omega^1_{\cur{R}_{\ekd}}\ar[r]  & M\otimes \Omega^1_{\cur{R}^u_{\ekd}} \ar[r]^{\mathrm{tr}} & M\otimes \Omega^1_{\cur{R}_{\ekd}}
}
$$
commutes, and the composite maps
\begin{align*}
M\rightarrow &M\otimes \cur{R}_{\ekd}^u\rightarrow M \\
M\otimes \Omega^1_{\cur{R}_{\ekd}}\rightarrow  &M\otimes \Omega^1_{\cur{R}^u_{\ekd}} \rightarrow M\otimes \Omega^1_{\cur{R}_{\ekd}}
\end{align*} are both multiplication by $n=[\cur{R}:\cur{R}_{\ekd}^\mathrm{int}]$. Hence there exists a projector $H^i(M\otimes \cur{R}_{\ekd}^{u})\rightarrow H^i(M \otimes \cur{R}_{\ekd}^{u})$ whose image is exactly $H^i(M)$. Of course, similar considerations hold over $\cur{R}_{\ek}$, and hence we get a commutative diagram
$$
\xymatrix{
H^i( M \otimes \cur{R}_{\ekd}^{u} ) \otimes_{\ekd} \ek \ar[r]\ar[d]& H^i(M\otimes \cur{R}_{\ekd}^{u}) \otimes_{\ekd} \ek \ar[d] \\
H^i(M \otimes \cur{R}_{\ek}^{u} ) \ar[r] &  H^i(M\otimes \cur{R}_{\ek}^{u}) 
}
$$
where the images of the horizontal arrows are $H^i(M)\otimes_{\ekd}\ek$ and $H^i(M\otimes \cur{R}_{\ek})$ respectively. Since $M\otimes \cur{R}^u_{\ekd}$ is unipotent, the vertical arrows are isomorphisms, and hence we get an isomorphism
$$
H^i(M)\otimes_{\ekd}\ek \rightarrow H^i(M\otimes \cur{R}_{\ek})
$$
as required.
\end{proof}

\section{Finiteness of rigid cohomology for smooth curves}\label{finsmc}

Now armed with a suitable version of the $p$-adic local monodromy theorem we can prove finite dimensionality of $H^i_\rig(X/\ekd,\sh{E})$ for smooth curves $X$, though perhaps not in the expected manner. One might expect to prove this by attaching a version of the Robba ring to every missing point of $X$ as in \S7.3 of \cite{crew1}, and deducing finite dimensionality directly, but we will not do this. Instead, we will prove finite dimensionally for $\A^1$ directly using the $p$-adic monodromy theorem, and then deduce it for more general smooth curves by locally pushing forward via a finite \'{e}tale map to $\A^1$. 

Our first task is therefore to use the monodromy theorem to prove finite dimensionality and base change with coefficients on the affine line.

\begin{theo} \label{finitesheaf} Let $\sh{E}\in F\text{-}\mathrm{Isoc}^\dagger(\A^1_{k\lser{t}}/\ekd)$. Then
$$ H^i_\rig(\A^1_{k\lser{t}}/\ekd,\sh{E}) $$ is finite dimensional for all $i$ and the natural map
$$
H^i_\rig(\A^1_{k\lser{t}}/\ekd,\sh{E})\otimes_{\ekd}\ek\rightarrow H^i_\rig(\A^1_{k\lser{t}}/\ek,\hat{\sh{E}})
$$ is an isomorphism.
\end{theo}

First of all we will need to reinterpret this result \`{a} la Monsky-Washnitzer, in order to be able to use the results of the previous section. The frame we will choose is the obvious one $(\A^1_{k\lser{t}},\P^1_{k\pow{t}},\widehat{\P}^1_{\cur{V}\pow{t}})$, and we will let $\varphi$ denote any lift to this frame of the $q$-power Frobenius on $\P^1_{k\pow{t}}$ compatible with the chosen Frobenius $\sigma$ on $\cur{V}\pow{t}$. Let $\P^{1,\mathrm{an}}_{S_K}=(\widehat{\P}^1_{\cur{V}\pow{t}})_K$ denote the analytic projective line over $S_K$, with co-ordinate $x$, say. 

\begin{lem} \label{a1cofinal} The open sets $$V_m:= \left\{ \left.P \in \P^{1,\mathrm{an}}_{S_K} \;\right|\; v_P(\pi^{-1}t^m)\leq 1, v_P(\pi x^m)\leq 1\right\}$$ are an cofinal system of neighbourhoods of $]\A^1_{k\lser{t}}[_{\widehat{\P}^1_{\cur{V}\pow{t}}}$ inside $\P^{1,\mathrm{an}}_{S_K}$.
\end{lem}

\begin{proof} This is Proposition 2.6 of \cite{rclsf1}.
\end{proof}

Note that each $V_m$ is affinoid, corresponding to the adic spectrum of the ring
$$
\cur{E}_{m}\tate{r^{-1/m}x} := \frac{S_K\tate{x,S_1,S_2}}{(\pi S_1-t^m,S_2-\pi x^m)}.
$$
We let $\ekd\weak{x}=\mathrm{colim}_m\cur{E}_{m}\tate{r^{-1/m}x}$, this has the alternative description as the ring of power series $\sum_if_ix^i$ such that there exists $\eta<1$ and $r>1$ with $f_i\in\cur{E}_\eta$ for all $i$ and $\Norm{f_i}_\eta r^i\rightarrow 0$ as $i\rightarrow \infty$. In other words, $\ekd\weak{x}=\mathrm{colim}_{\eta<1,\rho>1}\cur{E}_\eta\tate{\rho^{-1}x}$. We let $\cur{O}_{\ekd}\weak{x}$ denote the ring of power series $\sum_if_ix^i$ in $\ekd\weak{x}$ such that each $f_i\in\cur{O}_{\ekd}$, thus $\cur{O}_{\ekd}\weak{x}/\pi\cong k\lser{t}[x]$ is the polynomial ring in one variable over $k\lser{t}$.

\begin{defn} A Frobenius on $\ekd\weak{x}$ is a homomorphism which is $\sigma$-linear over $\ekd$, preserves $\cur{O}_{\ekd}\weak{x}$, and induces the absolute $q$-power Frobenius on $k\lser{t}[x]$.
\end{defn}

Note that any choice of Frobenius on the frame $(\A^1_{k\lser{t}},\P^1_{k\pow{t}},\widehat{\P}^1_{\cur{V}\pow{t}})$ induces a compatible collection of Frobenii, all of which we will denote by $\sigma$, on the rings in the diagram
$$\xymatrix{ \ekd\weak{x} \ar[rr]^{x\mapsto y^{-1}}\ar[d] && \cur{R}_{\ekd}\ar[d] \\ \ek\weak{x} \ar[rr]^{x\mapsto y^{-1}} &&\cur{R}_{\ek}.}$$

\begin{defn} Let $\sigma$ be a Frobenius on $\ekd\weak{x}$, and let $\partial_x:\ekd\weak{x}\rightarrow \ekd\weak{x}$ be the derivation given by differentiation with respect to $x$. 
\begin{itemize}
\item A $\varphi$-module over $\ekd\weak{x}$ is a finite $\ekd\weak{x}$-module $M$ together with a Frobenius structure, that is an $\sigma$-linear map $$\varphi: M \rightarrow M$$which induces an isomorphism $M\otimes_{\ekd\weak{x},\sigma} \ekd\weak{x}\cong M$.

\item A $\nabla$-module over $\ekd\weak{x}$ is a finite $\ekd\weak{x}$-module $M$ together with a connection, that is an $\ekd$-linear map $$\nabla:M\rightarrow  M$$ such that $\nabla(fm)=\partial_x(f)m+f\nabla(m)$ for all $f\in\ekd\weak{x}$ and $m\in M$.
\item A $(\varphi,\nabla)$-module over $\ekd\weak{x}$ is a finite $\ekd\weak{x}$-module $M$ together with a Frobenius $\varphi$ and a connection $\nabla$, such that the diagram
$$
\xymatrix{
M \ar[r]^\nabla \ar[d]^{\varphi} & M \ar[d]^{\partial_x(\sigma(x))\varphi} \\
M \ar[r]^\nabla & M
}
$$
commutes.
\end{itemize}
\end{defn}

Now, suppose that $\sh{E}$ is an overconvergent $F$-isocrystal on $\A^1_{k\lser{t}}/\ekd$. Then for each $m\gg0$ there is a module with (integrable) connection $\sh{E}_m$ on $V_m$ giving rise to $\sh{E}$, and thus the direct limit over all the $\Gamma(V_m,\sh{E}_m)$ will be a module with connection over $\ekd\weak{x}$, that is a finitely generated $\ekd\weak{x}$-module $M$ together with a connection $\nabla:M\rightarrow M$. Exactly as for $\ekd\weak{x}$, the Frobenius structure on $\sh{E}$ gives rise to a Frobenius structure on $M$, thus we get a functor
$$
\sh{E}\mapsto M:=\Gamma( \P^{1,\mathrm{an}}_{S_K},\sh{E})
$$
from overconvergent $F$-isocrystals on $\A^1_{k\lser{t}}$ to $(\varphi,\nabla)$-modules over $\ekd\weak{x}$. 

Now let $$\ek\weak{x}=\mathrm{colim}_m\ek\tate{r^{-1/m}x}$$ denote the 1-dimensional Monsky-Washnitzer algebra over $\ek$. Since we have compatible Frobenii on $\ekd\weak{x}$ and $\ek\weak{x}$ we get a base extension functor $M\mapsto M':= M\otimes_{\ekd\weak{x}}\ek\weak{x}$ from $(\varphi,\nabla)$-modules over $\ekd\weak{x}$ to those over $\ek\weak{x}$. Letting $H^i(M')$ denote the cohomology of the complex $M'\overset{\nabla}{\rightarrow} M'$ there is then a base change morphism
$$
H^i(M)\otimes_{\ekd}\ek\rightarrow H^i(M').
$$

\begin{prop} For any $\sh{E}\in F\text{-}\mathrm{Isoc}^\dagger(\A^1_{k\lser{t}}/\ekd)$ with associated $(\varphi,\nabla)$-module $M$ over $\ekd\weak{x}$, there is an isomorphism
$$
H^i_\rig(\A^1_{k\lser{t}}/\ekd,\sh{E})\cong H^i(M)
$$
Moreover, if we let $M'=M\otimes_{\ekd\weak{x}}\ek\weak{x}$, then the base change morphism 
$$
H^i_\rig(\A^1_{k\lser{t}}/\ekd,\sh{E})\otimes_{\ekd}\ek\rightarrow H^i_\rig(\A^1_{k\lser{t}}/\ek,\hat{\sh{E}})
$$
can be identified with the base change morphism
$$
H^i(M)\otimes_{\ekd}\ek\rightarrow H^i(M').
$$
\end{prop}

\begin{proof} Once we have used the global differential $dx$ make the identification $$j^\dagger_{\A^1_{k\lser{t}}}\Omega^1_{\P^{1,\mathrm{an}}_{S_K}/{S_K}}\cong j^\dagger_{\A^1_{k\lser{t}}}\cur{O}_{\P^{1,\mathrm{an}}_{S_K}},$$
the complex $M\overset{\nabla}{\rightarrow} M$ is then just the global sections of the complex
$$
\sh{E}\rightarrow \sh{E}\otimes j^\dagger_{\A^1_{k\lser{t}}}\Omega^1_{\P^{1,\mathrm{an}}_{S_K}/{S_K}}.
$$ 
Hence for the first claim it suffices to prove that coherent $j^\dagger_{\A^1_{k\lser{t}}}\cur{O}_{\P^{1,\mathrm{an}}_{S_K}}$-modules are $\Gamma(\P^{1,\mathrm{an}}_{S_K},-)$ acyclic.

So let $\sh{F}$ be a coherent $j_{\A^1_{k\lser{t}}}^\dagger\cur{O}_{\P^{1,\mathrm{an}}_{S_K}}$-module, and let $j_m:V_m\rightarrow \P^{1,\mathrm{an}}_{S_K}$ denote the inclusion of the cofinal system of neighbourhoods from Lemma \ref{a1cofinal}. Since the tube $]\P^1_{k\pow{t}}[_{\widehat{\P}^1_{\cur{V}\pow{t}}}\cong \P^{1,\mathrm{an}}_{S_K}$ is quasi-compact, it follows from Lemma 1.15 of \cite{rclsf1} that 
$$
H^i(\P^{1,\mathrm{an}}_{S_K},\sh{F}) \cong \mathrm{colim}_{m\gg0} H^i(\P^{1,\mathrm{an}}_{S_K},j_{m*}\sh{F}_m)
$$
where $\sh{F}_m$ is a coherent $\cur{O}_{V_m}$-module inducing $\sh{F}$, for $m\gg0$. Since each $V_m$ is affinoid, the pushforward $j_{m*}$ is acyclic, so we have 
$$
\mathrm{colim}_{m\gg0} H^i(\P^{1,\mathrm{an}}_{S_K},j_{m*}\sh{F}_m)\cong \mathrm{colim}_{m\gg0} H^i(V_m,\sh{F}_m).
$$
Again, since $V_m$ is affinoid, we have
$$
H^i(V_m,\sh{F}_m)=0
$$
for $i>0$, and the first claim is proven.

Since an entirely similar argument applies to show that 
$$
 H^i_\rig(\A^1_{k\lser{t}}/\ek,\hat{\sh{E}})
$$
can be computed in terms of the global sections of $\hat{\sh{E}}\otimes j^\dagger_{\A^1_{k\lser{t}}}\Omega^*_{\P^{1,\mathrm{an}}_{\ek}/\ek}$, which is just the restriction of $\sh{E}\otimes j^\dagger_{\A^1_{k\lser{t}}}\Omega^*_{\P^{1,\mathrm{an}}_{S_K}/S_K}$ to the open subset 
$$\P^{1,\mathrm{an}}_{\cur{E}_K}\subset \P^{1,\mathrm{an}}_{S_K},$$
to prove the second claim it suffices to show that for any coherent $j_{\A^1_{k\lser{t}}}^\dagger\cur{O}_{\P^{1,\mathrm{an}}_{S_K}}$-module $\sh{F}$, there is an isomorphism
$$
\Gamma(\P^{1,\mathrm{an}}_{\cur{E}_K},\sh{F}) \cong \Gamma(\P^{1,\mathrm{an}}_{S_K},\sh{F})\otimes_{\ekd\weak{x}}\ek\weak{x}.
$$
Let $W_m:=V_m\cap \P^{1,\mathrm{an}}_{\ek}$, these form a cofinal system of neighbourhoods of $]\A^1_{k\lser{t}}[_{ \P^{1,\mathrm{an}}_{\ek}}$ inside $\P^{1,\mathrm{an}}_{\ek}$. Since $W_m$ is an open affinoid subset of the affinoid $V_m$, we have that
$$
\Gamma(W_m,\sh{F}_m)\cong \Gamma(V_m,\sh{F}_m)\otimes_{\Gamma(V_m,\cur{O}_{V_m})} \Gamma(W_m,\cur{O}_{W_m}).
$$
for any coherent $\cur{O}_{V_m}$-module $\sh{F}_m$, and the claim follows from taking the colimit as $m\rightarrow \infty$.
\end{proof}

Hence we can rephrase Theorem \ref{finitesheaf} as follows.

\begin{theo} \label{finitemodule} Let $M$ be a $(\varphi,\nabla)$-module over $\ekd\weak{x}$. Then $H^i(M)$ is finite dimensional over $\ekd$, and if $M'=M\otimes_{\ekd\weak{x}}\ek\weak{x}$, then the base change morphism
$$
H^i(M)\otimes_{\ekd}\ek\rightarrow H^i(M')
$$
is an isomorphism.
\end{theo}

As in the previous section with the proof of a version of the monodromy theorem, the proof will be inspired by the proof of generic finiteness and base change for relative Mosnky-Washnitzer cohomology by Kedlaya in \cite{kedlayafiniteness}, where for a dagger algebra $A$ he shows generic finiteness of the pushforward of a $(\varphi,\nabla)$ module via $A\rightarrow A\weak{x}$ by descending from the completion of the fraction field of $A$. Again, in our situation the `dagger algebra' is $\ekd$, and the completion of its fraction field is $\ek$, and the idea behind the proof easily adapts to our situation. Before we give the proof, however, we will need to know that $(\varphi,\nabla)$-modules over $\ekd\weak{x}$ are free, thus enabling us to apply the results of the previous section to their base change to $\cur{R}_{\ekd}$. This will need building up to.

\begin{defn} \begin{enumerate}  \item If $f=\sum_i f_ix^i \in \ek\tate{x}$ then we say that $f$ has order $k$ if $\Norm{f_i} \leq \Norm{f_k}$ for $i\leq k$ and $\Norm{f_i}<\Norm{f_k}$ for $i>k$, where $\Norm{\cdot}$ is the $\pi$-adic norm on $\ek$. If $A\subset \ek\tate{x}$ is a subring (for example $A=\ekd\weak{x}$) then we say $f\in A$ has order $k$ if it does so in $\ek\tate{x}$.
\item If $f=\sum f_ix^i \in \ek\weak{x}$ and $\rho>1$ then we say $f$ has $\rho$-order $k$ if $f\in\ek\tate{\rho^{-1}x}$, $\Norm{f_i}\rho^i \leq \Norm{f_k}\rho^k$ for $i\leq k$ and $\Norm{f_i}\rho^i<\Norm{f_k}\rho^k$ for $i>k$. Again, we will also use this terminology for subrings of $\ek\weak{x}$. 
\item If $f=\sum f_ix^i \in \ekd\weak{x}$ and $\eta<1,\rho>1$ then we say $f$ has $(\eta,\rho)$-order $k$ if $f\in\cur{E}_\eta\tate{\rho^{-1}x}$, $\Norm{f_i}_\eta\rho^i \leq \Norm{f_k}_\eta\rho^k$ for $i\leq k$ and $\Norm{f_i}_\eta\rho^i<\Norm{f_k}_\eta\rho^k$ for $i>k$. Here $\Norm{\cdot}_\eta$ is the natural norm on $\cur{E}_\eta$.  \end{enumerate}
\end{defn}

\begin{lem} \begin{enumerate} \item Let $f\in \ek\weak{x}$, and suppose that $f$ has order $k$. Then there exists $\rho_0>1$ such that for all $1<\rho\leq \rho_0$, $f$ has $\rho$-order $k$.
\item Let $f\in \ekd\weak{x}$, and suppose that $f$ has order $k$. Then there exists $\eta<1,\rho>1$ such that $f$ has $(\eta,\rho)$-order $k$. 
\end{enumerate}
\end{lem}

\begin{proof} \begin{enumerate} \item First note that since the $\rho$-order of $f$ can only decrease as $\rho$ decreases, it suffices to prove that there exits some $\rho$ such that $f$ has $\rho$-order $k$. Choose $\rho$ such that $f\in \ek\tate{\rho^{-1}x}$ and choose $i_0 \geq  k$ such that $i\geq i_0\Rightarrow \Norm{f_i}\rho^i < \Norm{f_k}$. Since $\Norm{f_k}\geq \Norm{f_i}$ for all $i\leq k$, it follows that for any $\rho'>1$, $\Norm{f_k}\rho'^s\geq \Norm{f_i}\rho'^i$ for $i\leq k$, and if $1<\rho'<\rho$ and $i\geq i_0$, then $\Norm{f_i}\rho'^i < \Norm{f_i}\rho^i < \Norm{f_k} < \Norm{f_k}\rho'^k$. Hence it suffices to find $1 < \rho' < \rho$ such that $\Norm{f_i}\rho'^i < \Norm{f_k}\rho'^k$ for $k < i \leq i_0$. But now just taking $\rho'< \mathrm{min}_{k< i \leq i_0} \sqrt[i-k]{\frac{\Norm{f_k}}{\Norm{f_i}}}$ will do the trick, since $\Norm{f_i}<\Norm{f_k}$ for $k<i\leq i_0$.

\item Choose $\rho_0>1$ such that for all $1<\rho\leq\rho_0$, $f$ has $\rho$-order $k$, after possibly decreasing $\rho_0$ we may choose $\eta_0$ such that $f\in\cur{E}_{\eta_0}\tate{\rho_0^{-1}x}$, that is $\Norm{f_i}_{\eta_0} \rho_0^i\rightarrow 0$. Choose $i_0$ such that $i\geq i_0\Rightarrow \Norm{f_i}_{\eta_0} \rho_0^i<\Norm{f_k}$, thus for all $\eta_0\geq\eta<1$, all $1<\rho\leq \rho_0$ and all $i\geq i_0$ we have $\Norm{f_i}_\eta\rho^i < \Norm{f_k}_\eta\rho^k$. 

Now, since $\Norm{f_i}\leq \Norm{f_k}$ for all $i\leq k$, then for all $\epsilon>0$ there exists some $\eta$ such that $\Norm{f_i}_\eta \leq \Norm{f_k}_\eta+\epsilon$ for all $i\leq k$ (since there are only a finite number of such $i$). Hence by taking $\epsilon$ sufficiently small, we can find $\eta_0\leq \eta_1<1$ and $ 1< \rho\leq \rho_0$ such that $\Norm{f_i}_\eta\rho^i \leq \Norm{f_k}_\eta\rho^k$ for $i\leq k$ and $\eta_1\leq \eta<1$. Similarly, since $\Norm{f_i}\rho^i < \Norm{f}\rho^k$ for all $k < i\leq i_0$, we can find some $\eta$ such that $\Norm{f_i}_\eta\rho^i < \Norm{f}_\eta\rho^k$ for all $k < i\leq i_0$. Hence $\Norm{f_i}_\eta\rho^i < \Norm{f}_\eta\rho^k$ for $i>k$ and $\Norm{f_i}_\eta\rho^i \leq\Norm{f}_\eta\rho^k$ for $i\leq k$, i.e. $f$ has $(\eta,\rho)$-order $k$.
\end{enumerate}
\end{proof}

\begin{lem} \label{weier} \begin{enumerate} \item Suppose that $f,g\in \ekd\weak{x}$ and $g$ has order $k$. Then there exist unique $q\in \ekd\weak{x}$ and $r\in \ekd[x]$ of degree $<k$ such that $f=qg+r$. 
\item Let $f\in \ekd\weak{x}$. Then there exists a polynomial $h\in\ekd[x]$ and a unit $u\in(\ekd\weak{x})^\times$ such that $f=uh$.
\end{enumerate} 
\end{lem}

\begin{proof} \begin{enumerate} \item The Weierstrass Division Lemma in $\ek\tate{x}$ tells us that there exist unique $q\in\ek\tate{x}$ and $r\in \ek[x]$ of degree $<k$ with $f=qg+r$, we need to show that in fact $q\in\ekd\weak{x}$ and $r\in\ekd[x]$. 

To prove this, choose $\eta,\rho$ such that $g$ has $(\eta,\rho)$-order $k$ and define the norm $\Norm{\cdot}_{\eta,\rho}$ on $\cur{E}_\eta\tate{\rho^{-1}x}$ by
$$ \Norm{\sum_ih_ix^i} = \sup_i \left\{\Norm{h_i}_\eta\rho^i\right\}.
$$
Then $\cur{E}_\eta\tate{\rho^{-1}x}$ is complete with respect to this norm, and after scaling by some constant in $K$ we may assume that $\Norm{g}_{\eta,\rho}=1$. Then exactly as in the proof of the usual Weierstrass Division Lemma (see for example Theorem 8, \S2.2 of \cite{bosch}), since $g$ has $(\eta,\rho)$-order $k$, we can find a sequence of elements $f_i,q'_i\in\cur{E}_\eta\tate{\rho^{-1}x}$, $r'_i\in\cur{E}_\eta[x]$ of degree $<k$, such that 
\begin{align*} f_i &= q'_ig+r'_i+f_{i+1} \\
\Norm{f_i}_{\eta,\rho},&\Norm{q'_i}_{\eta,\rho},\Norm{r'_i}_{\eta,\rho} \leq \epsilon^i\Norm{f}_{\eta,\rho}
\end{align*}
where $\epsilon = \max_{i>s} \left\{\Norm{g_i}_\eta\rho^i \right\} <1$. Hence the series $\sum_i q'_i$ and $\sum_i r'_i$ converge to elements $q'\in \cur{E}_\eta\tate{\rho^{-1}x}$ and $r'\in \cur{E}_\eta[x]$ of degree $<k$ such that 
$$ f= q'g+r'. 
$$
Thus by the uniqueness of such a division inside $\ek\tate{x}$, we get $q=q'$ and $r=r'$, or in other words $q\in \ekd\weak{x}$ and $r\in \ekd[x]$. 

\item Exactly as in the proof of the Weierstrass Preparation Lemma (see for example Corollary 9, \S2.2 of \cite{bosch}), for any $f\in \ekd\weak{x}$ we can write $f=uh$ with $h\in \ekd[x]$ and $u\in\ekd\weak{x}$ such that $u$ is a unit in $\ek\tate{x}$, we need to show that in fact $u$ is a unit in $\ekd\weak{x}$. After scaling by some element of $K$, we may assume that $\Norm{u}=1$, and thus by Corollary 4, \S2.2 of \cite{bosch}, that $u$ has order $1$. Hence there exits some $(\eta,\rho)$ such that $u$ has $(\eta,\rho)$-order $1$. So in the usual way we can write $u=1-v$ with $\Norm{v}_{\eta,\rho}<1$ and hence $\sum_i v^i$ is an inverse for $u$ in $\cur{E}_\eta\tate{\rho^{-1}x}$. 
\end{enumerate}
\end{proof}

\begin{lem} \label{free} Let $M$ be a $(\varphi,\nabla)$-module over $\ekd\weak{x}$. Then $M$ is free as a $\ekd\weak{x}$-module.
\end{lem}

\begin{proof} Lemma \ref{weier} above implies that every ideal in $\ekd\weak{x}$ is generated by some polynomial $h\in\ekd[x]$, and hence in particular $\ekd\weak{x}$ is a PID. Thus to prove that $M$ is free it suffices to prove that its torsion submodule $M^\mathrm{tor}$ is zero. Using the structure theorem for modules over a PID, it suffices to prove that for every irreducible element $f$ of $\ekd\weak{x}$, the $f$-power torsion submodule $M[f^\infty]$ is zero. Again using Lemma \ref{weier}, we may assume that $f$ is in fact an irreducible polynomial over $\ekd$, which remains irreducible in $\ekd\weak{x}$.

So let $m\in M[f^\infty]$, so that $f^km=0$ for some $k$. Then by the Leibniz rule, $kf^{k-1}f'm+f^k\nabla(m)=0$, and hence $f^{k+1}\nabla(m)=0$, i.e. $\nabla(m)\in M[f^\infty]$. Moreover, if we choose $k$ to be such that $f^k$ annihilates $M[f^\infty]$, but not $f^{k-1}$, then we must have that $f^k\nabla(m)=0$, and hence that $kf^{k-1}f'm=0$. Since this holds for all $m$, we conclude that $kf^{k-1}f'$ annihilates $M[f^\infty]$, and hence by the choice of $k$, we must have that $f^k$ divides $kf^{k-1}f'$. Since $f$ is irreducible, it thus follows that $f$ must divide $f'$. 

Since $f$ is irreducible in $\ekd[x]$, we know that we can write $\lambda f+\mu f'=1$ for some polynomials $\lambda,\mu$, and hence, since $f$ divides $f'$, it follows that $f$ must also divide $1$, i.e. $f$ is a unit and hence $M[f^\infty]=0$. 
\end{proof}

\begin{proof}[Proof of Theorem \ref{finitemodule}] There is an Frobenius-compatible embedding $\ekd\weak{x}\rightarrow \cur{R}_{\ekd}$ given by $x\mapsto y^{-1}$, and we let $\cur{Q}_{\ekd}$ denote the quotient. The snake lemma applied to the diagram with exact rows
$$
\xymatrix{ 0\ar[r]  & M \ar[r]\ar[d]^{\nabla} & M\otimes \cur{R}_{\ekd} \ar[r]\ar[d]^{\nabla^\mathrm{loc}} & M\otimes \cur{Q}_{\ekd} \ar[r] \ar[d]^{\nabla^\mathrm{qu}} & 0  \\ 0\ar[r] & 
M \ar[r] & M\otimes \cur{R}_{\ekd} \ar[r] & M\otimes \cur{Q}_{\ekd}  \ar[r] & 0
}
$$
induces a long exact sequence
\begin{align*}
0\rightarrow H^0(M)&\rightarrow H^0(M\otimes \cur{R}_{\ekd}) \rightarrow H^0(M\otimes \cur{Q}_{\ekd}) \\&\rightarrow H^1(M)\rightarrow H^1(M\otimes \cur{R}_{\ekd}) \rightarrow H^1(M\otimes \cur{Q}_{\ekd})\rightarrow 0.
\end{align*}
There is a similar long exact sequence associated to $M'=M\otimes \ek\weak{x}$ coming from the exact sequence
$$
 0 \rightarrow \ek\weak{x}\rightarrow \cur{R}_{\cur{E}_K}\rightarrow \cur{Q}_{\ek}\rightarrow 0
$$
as in 7.3.2 of \cite{kedlayafiniteness}, and these exact sequences are compatible with base change, in that there are morphisms
\begin{align*}
H^i(M)\otimes_{\ekd}\ek &\rightarrow H^i(M') \\
H^i(M\otimes \cur{R}_{\ekd})\otimes_{\ekd}\ek &\rightarrow H^i(M'\otimes \cur{R}_{\ek}) \\
H^i(M\otimes \cur{Q}_{\ekd})\otimes_{\ekd}\ek &\rightarrow H^i(M'\otimes \cur{Q}_{\ek})
\end{align*}
which form a commutative diagram of long exact sequences. By Lemma \ref{free}, $M$ is free as an $\ekd\weak{x}$-module, and so we can apply Theorem \ref{robbabase} which tells us that the maps $H^i(M\otimes \cur{R}_{\ekd})\otimes_{\ekd}\ek\rightarrow H^i(M\otimes \cur{R}_{\ek})$ are isomorphisms.  

\begin{claim} The map $$H^0(M\otimes \cur{Q}_{\ekd})\otimes_{\ekd}\ek\rightarrow H^0(M'\otimes \cur{Q}_{\ek})$$ is injective, and the map $$H^1(M)\otimes_{\ekd}\ek\rightarrow H^1(M')$$ is surjective.
\end{claim}

\begin{proof}[Proof of Claim] For the claim about injectivity, it suffices to show that the natural map $\cur{Q}_{\ekd}\otimes_{\ekd}\ek\rightarrow \cur{Q}_{\ek}$ is injective, which boils down to the diagram
$$
\xymatrix{
\ekd\weak{x}\ar[r]\ar[d] & \cur{R}_{\ekd} \ar[d] \\
\ek\weak{x} \ar[r] & \cur{R}_{\ek}
}
$$ being Cartesian. This follows straight from the definitions, since the inclusions
\begin{align*}
\ekd\weak{x} &\rightarrow \cur{R}_{\ekd} \\
\ek\weak{x} &\rightarrow \cur{R}_{\ek}
\end{align*}
identify $\ekd\weak{x}$ and $\ek\weak{x}$ with $\cur{R}_{\ekd}\cap \ekd\pow{y^{-1}}$ and $\cur{R}_{\ek} \cap \ek\pow{y^{-1}}$ respectively. 
For the claim about surjectivity, topologise $M'$ with the fringe topology, arising from the direct limit topology on $\ek\weak{x}$ as in Definition 2.3.7 of \cite{kedlayafiniteness}. Then since $M\otimes_{\ekd}\ek$ is dense inside $M'$ for this topology, it follows that the map
$$
H^1(M)\otimes_{\ekd}\ek \rightarrow H^1(M')
$$ 
has dense image for the induced topology on $M'$. Since the fringe topology is Hausdorff, and the image of 
$$
\nabla:M'\rightarrow M'
$$
is closed by Proposition 8.4.4 of \emph{loc. cit.}, it follows that $H^1(M')$ is also Hausdorff for the induced topology. Since $H^1(M')$ is finite dimensional over $\ek$, any dense subspace must therefore be equal to the whole space, and hence the claim follows.
\end{proof}

Hence we can apply Lemma 7.5.3 of \cite{kedlayafiniteness} to conclude that the the maps $H^i(M)\otimes_{\ekd}\ek\rightarrow H^i(M')$ must be isomorphisms, and it then follows that each $H^i(M)$ is finite dimensional over $\ekd$.
\end{proof}

We are now in a position to deduce finite dimensionality of $H^i_\rig(X/\ekd,\sh{E})$ for smooth curves $X$ from finiteness for $\A^1$ just proven. The result we will therefore be spending the rest of this section proving is the following.

\begin{theo} \label{finitecurves} Let $X/k\lser{t}$ be a smooth curve, and let $\sh{E}$ be an overconvergent $F$-isocrystal on $X/\ekd$, with associated overconvergent $F$-isocrystal $\hat{\sh{E}
}$ on $X/\ek$. Then the cohomology groups
$$
H^i_\rig(X/\ekd,\sh{E})
$$
are finite dimensional, and the base change maps
$$ H^i_\rig(X/\ekd,\sh{E})\otimes_{\ekd}\ek\rightarrow H^i_\rig(X/\ek,\hat{\sh{E}})
$$
are isomorphisms.
\end{theo}

Note that since smooth curves are quasi-projective, there always exists an embedding into a smooth and proper frame over $\cur{V}\pow{t}$. Exactly as the general strategy for proving finiteness in \cite{kedlayafiniteness}, we will prove this by descending to $\A^1$ using (the one dimensional case of) the main result from \cite{kedetale}, namely the following.

\begin{theo}[\cite{kedetale}, Theorem 1] \label{finet} Let $X/k$ be a smooth projective curve over a field $k$ of characteristic $p>0$, and $D\subset X$ any non-empty divisor . Then there exists a finite morphism $f:X\rightarrow \P^1_k$ such that $D\cap f^{-1}(\infty)=\emptyset$ and $f$ is \'{e}tale away from $\infty$.
\end{theo}

For this to be useful for us we will need to know that we can `lift' a finite \'{e}tale map between curves over $k\lser{t}$-schemes to characteristic zero. Actually, since we will first need to compactify over $k\pow{t}$, the lifting problem is somewhat subtle, and we will need to make extensions of the ground field $k\lser{t}$ to ensure that we can pick `sufficiently nice' models over $k\pow{t}$. Our first result therefore will tell us that we may make such finite extensions with impunity. So let us suppose that we have a finite separable extension $F\cong l\lser{u}$ of $k\lser{t}$, with associated finite extensions $S_K^F/S_K$, $\cur{E}_K^{\dagger,F}/\ekd$ and $\cur{E}_K^F/\ek$. Recall from \S5 of \cite{rclsf1} that in this situation we have a commutative diagram of base extension functors
$$ \xymatrix{ \mathrm{Isoc}^\dagger(X/\ekd) \ar[rr]^{\sh{E}\mapsto \sh{E}_F}\ar[d]^{\sh{E}\mapsto \hat{\sh{E} }} && \mathrm{Isoc}^\dagger(X_F/\cur{E}_K^{\dagger,F})\ar[d]^{\sh{E}\mapsto \hat{\sh{E} }} \\
\mathrm{Isoc}^\dagger(X/\ek)\ar[rr]^{\sh{E}\mapsto \sh{E}_F} && \mathrm{Isoc}^\dagger(X_F/\cur{E}_K^F).
}
$$

\begin{lem} \label{basefinite} Let $X/k\lser{t}$ be an embeddable variety, and $\sh{E}\in\mathrm{Isoc}^\dagger(X/\ekd)$. The natural base change morphisms (see \S5 of \cite{rclsf1})
\begin{align*}
H^i_\rig(X/\ekd,\sh{E})\otimes_{\ekd} \cur{E}_K^{\dagger,F}&\rightarrow H^i_\rig(X_F/\cur{E}_K^{\dagger,F},\sh{E}^F) \\
H^i_\rig(X/\ek,\hat{\sh{E}})\otimes_{\ek} \cur{E}_K^{F}&\rightarrow H^i_\rig(X_F/\cur{E}_K^{F},\hat{\sh{E}}^F)
\end{align*}
are isomorphisms.
\end{lem}

\begin{proof} Let $\mathbb{D}^b_F$ denote the rigid space $\mathrm{Spa}(S_K^F)$, and let $(]Y[_\mathfrak{P})_F$ denote $]Y[_\mathfrak{P}\times_{\mathbb{D}^b_K} \mathbb{D}^b_F$, with $f:(]Y[_\mathfrak{P})_F\rightarrow ]Y[_\mathfrak{P}$ the projection. It is easy to see that there is a natural isomorphism of complexes
$$ f^*( \sh{E} \otimes \Omega^*_{]Y[_\mathfrak{P}/S_K}) \cong \sh{E}_F \otimes \Omega^*_{(]Y[_\mathfrak{P})_F/S_K^F}
$$
and this together with the fact that for any sheaf $E$ on $]Y[_\mathfrak{P}$ the base change morphism
$$ H^i(]Y[_\mathfrak{P},E) \otimes_{S_K} S_K^F \rightarrow H^i((]Y[_\mathfrak{P})_F,f^*E)
$$
is an isomorphism implies that $H^i_\rig(X/\ekd,\sh{E})\otimes_{\ekd} \cur{E}_K^{\dagger,F} \cong H^i_\rig(X_F/\cur{E}_K^{\dagger,F},\sh{E}^F)$. Of course, an entirely similar argument shows that
$$
H^i_\rig(X/\ek,\hat{\sh{E}})\otimes_{\ek} \cur{E}_K^{F}\rightarrow H^i_\rig(X_F/\cur{E}_K^{F},\hat{\sh{E}}^F)
$$
is an isomorphism
\end{proof}

Thus at any point during the proof of Theorem \ref{finitecurves}, we may always make a finite separable extension of the ground field $k\lser{t}$. We can therefore lift finite \'{e}tale maps to characteristic zero using the following proposition. 

\begin{prop} \label{forlift} Let $f:X\rightarrow \P^1_{k\lser{t}}$ be a finite morphism as in Theorem \ref{finet}, and let $U=f^{-1}(\A^1_{k\lser{t}})$. Then after replacing $k\lser{t}$ by a finite separable extension, there exists a $p$-adic formal scheme $\mathfrak{X}$, flat and proper over $\mathrm{Spf}(\cur{V}\pow{t})$, and commutative diagram
$$
\xymatrix{ U \ar[r]\ar[d]^f& \mathfrak{X} \ar[d]^w  \\
\A^1_{k\lser{t}} \ar[r] & \widehat{\P}^1_{\cur{V}\pow{t}}
}
$$ such that $w$ is \'{e}tale in a neighbourhood of $U$. \end{prop}

\begin{proof} By the semistable reduction theorem for curves, after making a finite separable extension of $k\lser{t}$ we may choose a model $\overline{X}\rightarrow \P^1_{k\pow{t}}$ such that $\overline{X}$ is semistable. Since $\overline{X}$ is semistable, it is a local complete intersection over $A$, and thus, since $\P^1_{k\pow{t}}$ is smooth over $k\pow{t}$,  the morphism $\overline{X}\rightarrow \P^1_{k\pow{t}}$ is a local complete intersection. Moreover, there exist finitely many points $\{P_i\}$ of the generic fibre $X$ and $\{Q_i\}$ of the special fibre $\overline{X}_0$ such that $\overline{X}$ is \'{e}tale over $\P^1_{k\pow{t}}$ away from the closed \emph{affine} subscheme $\overline{\{P_i\}}\cup \{Q_i\}$ of $\overline{X}$. Hence we may inductively apply Lemma \ref{lift} below.
\end{proof}

\begin{lem} \label{lift} Let $R_0$ be a ring, $X_0$ a flat curve over $R_0$, and $f_0:X_0\rightarrow \P^1_{R_0}$ a l.c.i. morphism. Suppose that there exists a closed affine subscheme $Z_0\subset X_0$, with open complement $U_0$, such that $f_0$ is \'{e}tale on $U_0$. Let $R_1\rightarrow R_0$ be a square zero extension. Then there exists a lifting of $X_0$ to a curve $X_1$, flat over $R_1$, and a lifting of $f_0$ to an l.c.i. morphism $f_1:X_1\rightarrow \P^1_{R_1}$. Moreover, if we let $U_1\subset X_1$ denote the open subscheme corresponding to $U_0$, then $f_1$ is \'{e}tale on $U_1$.\end{lem}

\begin{proof} We follow closely the proof of Lemma 8.3 in \cite{crew1}. The obstruction to the existence of a flat lifting is a class in $\mathrm{Ext}^2_{X_0}(L_{X_0/\P^1_{R_0}},\cur{O}_{X_0})$, where $L_{X_0/\P^1_{R_0}}$ is the relative cotangent complex. Since $X_0\rightarrow \P^1_{R_0}$ is l.c.i., $L_{X_0/\P^1_{R_0}}$ has perfect amplitude in $[-1,0]$. Thus $L_{X_0/\P^1_{R_0}}^\vee:=\mathbf{R}\mathrm{Hom}(L_{X_0/\P^1_{R_0}},\cur{O}_{X_0})$ has perfect amplitude in $[0,1]$, and in particular $\cur{H}^i(L_{X_0/\P^1_{R_0}}^\vee)=0$ for $i\neq0,1$. Since $f_0$ is \'{e}tale on $U_0$, it follows that $L_{X_0/\P^1_{R_0}}|_{U_0}=0$, and hence that $\cur{H}^i(L_{X_0/\P^1_{R_0}}^\vee)$ has support in $Z_0$ for $i=0,1$. Thus we have $$H^1(X_0,\cur{H}^1(L_{X_0/\P^1_{R_0}}^\vee))=0=H^2(X_0,\cur{H}^0(L_{X_0/\P^1_{R_0}}^\vee)),$$ hence $\mathrm{Ext}^2_{X_0}(L_{X_0/\P^1_{R_0}},\cur{O}_{X_0})$, and the obstruction must vanish. Since $X_1/R_1$ is flat, the fact that $f_1$ is l.c.i and \'{e}tale over $U_1$ follows from the same facts about $f_0$.
\end{proof}

This allows us to construct pushforwards of overconvergent $F$-isocrystals via a finite \'{e}tale morphism to $\A^1$, at least after making a finite separable extension of $k\lser{t}$, as follows. Let $f:X\rightarrow \P^1_{k\lser{t}}$ be a finite morphism as in Theorem \ref{finet}, then after making a finite separable extension of $k\lser{t}$ we may assume that there exists a morphism of smooth and proper frames
$$
w:(U,\overline{X},\mathfrak{X})\rightarrow (\A^1_{k\lser{t}},\P^1_{k\pow{t}},\widehat{\P}^1_{\cur{V}\pow{t}}),
$$
where $\overline{X}\cong\mathfrak{X}\otimes_{\cur{V}\pow{t}} k\pow{t}$, such that $\overline{X}$ is proper over $\P^1_{k\pow{t}}$ and $\frak{X}$ is \'{e}tale over $\widehat{\P}^1_{\cur{V}\pow{t}}$ in a neighbourhood of $U$. Note that we have $]\overline{X}[_{\mathfrak{X}}=\mathfrak{X}_{K}$ and $]\P^1_{k\pow{t}}[_{\widehat{\P}^1_{\cur{V}\pow{t}}}=\P^{1,\mathrm{an}}_{S_K}$, and $w$ induces a morphism of ringed spaces
$$
w_K: (\mathfrak{X}_{K},j_{U}^\dagger\cur{O}_{\mathfrak{X}_{K}})\rightarrow (\P^{1,\mathrm{an}}_{S_{K}},j_{\A^1}^\dagger\cur{O}_{\P^{1,\mathrm{an}}_{S_{K}}}).
$$ 
Since $w$ is \'{e}tale in a neighbourhood of $U$, if we let $U_m$ denote the standard neighbourhoods of $]\A^1_{k\lser{t}}[_{\widehat{\P}^1_{\cur{V}\pow{t}}}$ inside $\P^{1,\mathrm{an}}_{S_K}$, then for all $m\gg0$, the induced map $w_K^{-1}(U_m)\rightarrow U_m$ is finite \'{e}tale. Moreover, since $U=w^{-1}(\A^1_{k\lser{t}})$, it follows that the $w_K^{-1}(U_m)$ are exactly the standard neighbourhoods $U'_m$ of $]U[_{\mathfrak{X}}$ inside $\mathfrak{X}_K$. 

Now let $\sh{E}$ be an overconvergent isocrystal on $U/\ekd$, corresponding to a coherent $j_{U}^\dagger\cur{O}_{\mathfrak{X}_{K}}$-module with overconvergent connection which we will also denote by $\sh{E}$. Since $w_K$ is finite on $U_m$ for $m\gg0$, for any coherent $j_{U}^\dagger\cur{O}_{\mathfrak{X}_{K}}$-module $\sh{F}$, $\mathbf{R}w_{K*}\sh{F}\cong w_{K*}\sh{F}$ is a coherent $j_{\A^1}^\dagger\cur{O}_{\P^{1,\mathrm{an}}_{S_K}}$-module. Since $w_K$ is \'{e}tale on $U_m$ for $m\gg0$, it follows that $w_{K}^*j_{\A^1}^\dagger\Omega^1_{\P^{1,\mathrm{an}}_{S_{K}}/S_K}\cong j_{U}^\dagger\Omega^1_{\mathfrak{X}_{K}/S_K}$, and hence, if $\sh{E}$ is coherent $ j_{U}^\dagger \cur{O}_{\mathfrak{X}_{K}}$ with an overconvergent integrable connection, the projection formula implies that
\begin{align*}
\mathbf{R}w_{K*}(\sh{E}\rightarrow \sh{E}\otimes j_{U}^\dagger \Omega^1_{\mathfrak{X}_{K}/S_K}) &\cong w_{K*}(\sh{E}\rightarrow\sh{E}\otimes j_{U}^\dagger\Omega^1_{\mathfrak{X}_{K}/S_K}) \\
&\cong w_{K*}(\sh{E})\rightarrow w_{K*}(\sh{E})\otimes j_{\A^1}^\dagger\Omega^1_{\P^{1,\mathrm{an}}_{S_{K}}/S_K}
\end{align*}
which gives an integrable connection on the coherent $j_{\A^1}^\dagger\cur{O}_{\P^{1,\mathrm{an}}_{S_{K}}}$-module $w_{K*}(\sh{E})$.

\begin{prop} This connection is overconvergent, and the induced overconvergent isocrystal on $\A^1_{k\lser{t}}/\ekd$ depends only on $f:U\rightarrow \A^1_{k\lser{t}}$. This gives rise to a functor
$$
f_*:\mathrm{Isoc}^\dagger(U/\ekd)\rightarrow \mathrm{Isoc}^\dagger(\A^1_{k\lser{t}}/\ekd)
$$
adjoint to $f^*$ and commuting with finite extensions $F/k\lser{t}$. Moreover, for any $\sh{E}\in \mathrm{Isoc}^\dagger(U/\ekd)$ we have
$$
H^i_\rig(U/\ekd,\sh{E})\cong H^i_\rig(\A^1_{k\lser{t}}/\ekd,f_*\sh{E}).
$$
\end{prop}

\begin{proof} Suppose that $w:\mathfrak{X}\rightarrow \widehat{\P}^1_{\cur{V}\pow{t}}$ and $w':\mathfrak{X}'\rightarrow \widehat{\P}^1_{\cur{V}\pow{t}}$ are two choices of morphism of formal $\cur{V}\pow{t}$-schemes lifting $f:U\rightarrow \A^1_{k\lser{t}}$ as in Proposition \ref{forlift}. Then after replacing $\mathfrak{X}'$ by the fibre product $\mathfrak{X}'\times_{\widehat{\P}^1_{\cur{V}\pow{t}}}\mathfrak{X}$ and embedding $U$ via the diagonal morphism, we may assume that there exits a morphism $v:\mathfrak{X}'\rightarrow \mathfrak{X}$ which is finite and \'{e}tale in neighbourhood of $U$. Let $\overline{X}'$ denote the Zariski closure of $U$ in $\mathfrak{X}'\otimes_{\cur{V}\pow{t}} k\pow{t}$, and let $\sh{E},\sh{E}'$ denote the realisations of some isocrystal in $\mathrm{Isoc}^\dagger(U/\ekd)$ on $(U,\overline{X},\mathfrak{X})$ and $(U,\overline{X}',\mathfrak{X}')$ respectively. Again, $]\overline{X}'[_{\mathfrak{X}'}=\mathfrak{X}'_K$ and by the Strong Fibration Theorem, i.e. Proposition 4.1 of \cite{rclsf1}, we have that
$$
\mathbf{R}v_{K*}(\sh{E}'\otimes j_{U}^\dagger\Omega^*_{\mathfrak{X}'_K/S_K}) \cong \sh{E}\otimes j_{U}^\dagger\Omega^*_{\mathfrak{X}_K/S_K}
$$
from which we deduce the independence of the chosen lift $w$. 

To show that this connection is overconvergent we use Proposition 5.15 of \cite{rclsf1}. Let $x$ be the co-ordinate on $\widehat{\P}^1_{\cur{O}_F}$, so that $dx$ is a basis for $\Omega^1_{\widehat{\P}^1_{\cur{V}\pow{t}}/\cur{V}\pow{t}}$ in a neighbourhood of $\A^1$, as well as being a basis for $\Omega^1_{\mathfrak{X}/\cur{V}\pow{t}}$ in a neighbourhood of $U$. Also note that in this case the closed tubes $[-]_n$ are equal to $\mathfrak{X}_K$ and $\P^{1,\mathrm{an}}_{S_K}$ for all $n$. So let $n\geq0$ and let $U_m$ be one of the standard neighbourhoods of $]\A^1_{k\lser{t}}[_{\widehat{\P}^1_{\cur{V}\pow{t}}}$ inside $\P^{1,\mathrm{an}}_{S_K}$ such that $w_K^{-1}(U_m)\rightarrow U_m$ is finite \'{e}tale, $\sh{E}$ extends to a module with integrable connection on $w_K^{-1}(U_m)$, and
$$ \Norm{\frac{\partial^k_x(e)}{k!}} r^{-\frac{k}{n}}\rightarrow 0
$$
as $k\rightarrow 0$ for all $e\in \Gamma(w_K^{-1}(U_m),\sh{E})$. Since the Banach norm on $\Gamma(w_K^{-1}(U_m),\sh{E})$ as an $\cur{O}_{w_K^{-1}(U_m)}$-module is equivalent to the Banach norm on $\Gamma(w_K^{-1}(U_m),\sh{E})=\Gamma(U_m,w_{L*}\sh{E})$ as an $\cur{O}_{U_m}$-module, it follows that
$$ \Norm{\frac{\partial^k_x(e)}{k!}} r^{-\frac{k}{n}}\rightarrow 0
$$
for all $e\in \Gamma(U_m,w_{K*}\sh{E})$, in other words the connection on $w_{K*}(\sh{E})$ is overconvergent.

Finally, the fact that 
$$
\mathbf{R}\Gamma(\mathfrak{X}_K,-)\cong \mathbf{R}\Gamma(\P^{1,\mathrm{an}}_{S_K},\mathbf{R}w_{K*}(-))
$$
easily implies statement about the cohomology of the pushforward.\end{proof}

The next stage is to show that $f_*$ is compatible with Frobenius structures. Since $f:U\rightarrow \A^1_{k\lser{t}}$ is finite \'{e}tale, the relative Frobenius $F_{U/\A^1}:U\rightarrow U'$ is an isomorphism, where $U'$ is the base change of $U$ via the absolute ($q$-power) Frobenius of $\A^1_{k\lser{t}}$. Hence if $\varphi$ is a $\sigma$-linear lift of the $q$-power Frobenius on $\P^1_{k\pow{t}}$ to $\widehat{\P}^1_{\cur{V}\pow{t}}$ (such liftings always exist), then the base change $\mathfrak{X}':=\mathfrak{X}\times_{\widehat{\P}^1_{\cur{V}\pow{t}},\varphi}\widehat{\P}^1_{\cur{V}\pow{t}}$ of $\mathfrak{X}$ via $\varphi$ together with its natural map to $\mathfrak{X}$ is a lifting of the $q$-power Frobenius on $U$. 

Explicitly, let $\overline{X}'$ be the base change of $\overline{X}$ by the $q$-power Frobenius on $\P^1_{k\pow{t}}$, then the relative Frobenius $F_{U/\A^1}:U\rightarrow U'$ induces a closed immersion $U\rightarrow \overline{X}'$, and the induced diagram of frames
$$
\xymatrix{ (U,\overline{X}',\mathfrak{X}')\ar[r]^{\varphi'}\ar[d]^{w'} &  (U,\overline{X},\mathfrak{X}) \ar[d]^w \\
(\A^1_{k\lser{t}},\P^1_{k\pow{t}},\widehat{\P}^1_{\cur{V}\pow{t}}) \ar[r]^{\varphi} & (\A^1_{k\lser{t}},\P^1_{k\pow{t}},\widehat{\P}^1_{\cur{V}\pow{t}}) }
$$
is a lifting of the commutative diagram
$$
\xymatrix{
U \ar[r]\ar[d] & U \ar[d] \\
\A^1_{k\lser{t}} \ar[r] & \A^1_{k\lser{t}}  }
$$
where the horizontal morphisms are the $q$-power Frobenii. Hence by Lemma 5.22 of \cite{rclsf1}, the Frobenius pullback functors
\begin{align*}
F^*:\mathrm{Isoc}^\dagger(U/\ekd)&\rightarrow \mathrm{Isoc}^\dagger(U/\ekd) \\
F^*:\mathrm{Isoc}^\dagger(\A^1_{k\lser{t}}/\ekd)&\rightarrow \mathrm{Isoc}^\dagger(\A^1_{k\lser{t}}/\ekd)
\end{align*}
can be identified with $\varphi_K'^*$ and $\varphi_K^*$ respectively.

\begin{prop} \'{E}tale pushforward $f_{*}$ commutes with Frobenius pullback, in the sense that the diagram
$$
\xymatrix{ \mathrm{Isoc}^\dagger(U/\ekd)\ar[d]^{f_{*}}\ar[r]^{F^*} & \mathrm{Isoc}^\dagger(U/\ekd) \ar[d]^{f_{*}} \\
\mathrm{Isoc}^\dagger(\A^1_{k\lser{t}}/\ekd)\ar[r]^{F^*} & \mathrm{Isoc}^\dagger(\A^1_{k\lser{t}}/\ekd) 
}
$$
commutes up to natural isomorphism. 
\end{prop}

\begin{proof} For any coherent $ j_{U}^\dagger \cur{O}_{\mathfrak{X}_{K}}$-module $\sh{E}$ with an integrable connection, there is a natural base change morphism 
$$
\varphi_K^*w_{K*}\sh{E} \rightarrow w'_{K*}\varphi'^*_K\sh{E}
$$
which is horizontal for the induced connections, we must show that this morphism is an isomorphism. Of course, we may forget the connection and simply prove the corresponding statement for coherent $ j_{U}^\dagger \cur{O}_{\mathfrak{X}_{K}}$-modules.

We can given this a Monsky-Washnitzer interpretation as follows, we have $$\Gamma(\P^{1,\mathrm{an}}_{S_K},j_{\A^1}^\dagger\cur{O}_{\P^{1,\mathrm{an}}_{S_K}})=\ekd\weak{x}$$ and the Frobenius $\sigma$ induces a Frobenius $\sigma$ on $\ekd\weak{x}$. Then
$$ \mathrm{Coh}(j_{\A^1}^\dagger\cur{O}_{\P^{1,\mathrm{an}}_{S_K}}) \cong \mathrm{Mod}_{\mathrm{fg}}(\ekd\weak{x})
$$
and if we similarly let $A_K^\dagger= \Gamma(\mathfrak{X}_K,j_{U}^\dagger\cur{O}_{\mathfrak{X}_K})$, then we get an equivalence
$$ \mathrm{Coh}(j_{U}^\dagger\cur{O}_{\mathfrak{X}_K}) \cong \mathrm{Mod}_{\mathrm{fg}}(A_K^\dagger).
$$
There is a finite \'{e}tale morphism of rings $\ekd\weak{x}\rightarrow A_K^\dagger$ such that the push forward $f_{*}$ is identified with the forgetful functor 
$$  \mathrm{Mod}_{\mathrm{fg}}(A_K^\dagger)\rightarrow \mathrm{Mod}_{\mathrm{fg}}(\ekd\weak{x}).
$$
We can identify $\Gamma(\mathfrak{X}'_K,j_{U}^\dagger\cur{O}_{\mathfrak{X}'_K})$ with $$B^\dagger_K := A_K^\dagger\otimes_{\ekd\weak{x},\sigma} \ekd\weak{x}.$$ Then the base change morphism becomes the natural morphism
$$ M \otimes_{\ekd\weak{x},\sigma} \ekd\weak{x} \rightarrow M\otimes_{A_K^\dagger} B^\dagger_K
$$ 
associated to a finite $A_K^\dagger$-module $M$, which is an isomorphism.
\end{proof}

Finally, since we want to be able to prove base change theorems as well as finiteness theorems, we will need to know compatibility of $f_{*}$ with the \'{e}tale pushforward functor in `usual' rigid cohomology, as constructed for example in \S2.6 of \cite{tsuzukigysin}. Let $\cur{O}_{\ek}$ denote the ring of integers of $\ek$, and let $w:(U,\overline{X},\mathfrak{X})\rightarrow (\A^1_{k\lser{t}},\P^1_{k\pow{t}},\widehat{\P}^1_{\cur{V}\pow{t}})$ be a lifting of $U\rightarrow \A^1_{k\lser{t}}$ as above. Write $\overline{X}_{k\lser{t}}=\overline{X}\otimes_{k\pow{t}} k\lser{t}$ and $\hat{\mathfrak{X}}=\mathfrak{X}\otimes_{\cur{V}\pow{t}} \cur{O}_{\ek}$. Then
$$
w:(U,\overline{X}_{k\lser{t}},\hat{\mathfrak{X}})\rightarrow (\A^1_{k\lser{t}},\P^1_{k\lser{t}},\widehat{\P}^1_{\cur{O}_{\ek}})
$$
is a morphism of frames over $\cur{O}_{\ek}$ lifting $U\rightarrow \A^1_{k\lser{t}}$. Methods entirely similar to those used in \S1 of \cite{rclsf1} show that overconvergent isocrystals on $U/\ek$ (resp. $\A^1_{k\lser{t}}/\ek$) can be described as modules with overconvergent connection on $]\overline{X}_{k\lser{t}}[_{\hat{\mathfrak{X}}}=\hat{\mathfrak{X}}_{\ek}$ (resp. $]\P^1_{k\lser{t}}[_{\widehat{\P}^1_{\cur{O}_{\ek}}}=\P^{1,\mathrm{an}}_{\ek}$) as adic spaces, and that Tsuzuki's \'{e}tale pushforward functor can be described as taking the module with connection 
$$
\sh{E}\rightarrow \sh{E}\otimes j_{U}^\dagger\Omega^1_{\hat{\mathfrak{X}}_{\ek}/\ek} 
$$
to the module with connection corresponding to
$$
\mathbf{R}w_{K*}(\sh{E}\rightarrow \sh{E}\otimes j_{U}^\dagger\Omega^1_{\hat{\mathfrak{X}}_{\ek}/\ek} )\cong  w_{K*}(\sh{E}) \rightarrow w_{K*}(\sh{E})\otimes j_{\A^1}^\dagger\Omega^1_{\P^{1,\mathrm{an}}_{\ek}/\ek} 
$$
exactly as before. Since the natural morphisms of tubes appearing as the horizontal arrows in the diagram
$$\xymatrix{
\hat{\mathfrak{X}}_{\ek} \ar[r]\ar[d]^{w_{K}} &\mathfrak{X}_K\ar[d]^{w_K} \\
\P^{1,\mathrm{an}}_{\ek} \ar[r] & \P^{1,\mathrm{an}}_{S_K} }
$$
are open immersions, and the diagram is Cartesian, it follows immediately that the natural base change morphism arising from the diagram is an isomorphism, or in other words that the diagram
$$
\xymatrix{
\mathrm{Isoc}^\dagger(U/\ekd) \ar[r]^{\sh{E}\mapsto \hat{\sh{E} }}\ar[d]^{f_{*}} & \mathrm{Isoc}^\dagger(U/\ek) \ar[d]^{f_{*}} \\ \mathrm{Isoc}^\dagger(\A^1_{k\lser{t}}/\ekd) \ar[r]^{\sh{E}\mapsto \hat{\sh{E}}} & \mathrm{Isoc}^\dagger(\A^1_{k\lser{t}}/\ek)
}
$$
commutes up to natural isomorphism. Of course, an entirely similar argument shows that the same is true with Frobenius structures. We can summarise the results of this section so far as follows.

\begin{theo} \label{summ} Let $f:X\rightarrow \P^1_{k\lser{t}}$ be a finite morphism, \'{e}tale away from $\infty$, and let $U=f^{-1}(\A^1_{k\lser{t}})$. Then after a finite separable extension of $k\lser{t}$ there exists a pushforward functor
$$
f_{*}:F\text{-}\mathrm{Isoc}^\dagger(U/\ekd) \rightarrow F\text{-}\mathrm{Isoc}^\dagger(\A^1_{k\lser{t}}/\ekd)
$$
which commutes with further finite separable extensions $F/k\lser{t}$, and is such that 
$$
H^i_\rig(U/\ekd,\sh{E})\cong H^i_\rig(\A^1_{k\lser{t}}/\ekd,f_{*}\sh{E})
$$
for any $\sh{E}\in F\text{-}\mathrm{Isoc}^\dagger(U/\ekd)$. Moreover, the diagram
$$
\xymatrix{ F\text{-}\mathrm{Isoc}^\dagger(U/\ekd)
 \ar[r]\ar[d]^{f_{*}} & F\text{-}\mathrm{Isoc}^\dagger(U/\ek) \ar[d]^{f_{*}} \\ 
 F\text{-}\mathrm{Isoc}^\dagger(\A^1_{k\lser{t}}/\ekd)
 \ar[r] & F\text{-}\mathrm{Isoc}^\dagger(\A^1_{k\lser{t}}/\ek) 
}
$$
commutes up to natural isomorphism.
\end{theo}

We need one last lemma before we can prove Theorem \ref{finitecurves}. 

\begin{lem} \label{spectral} Let $X/k\lser{t}$ be an embeddable variety, $X=\cup_j X_j$ a finite open cover of $X$ and $\sh{E}\in \mathrm{Isoc}^\dagger(X/\ekd)$. For indices $J=\left\{j_1,\ldots,j_n\right\}$ of we set $X_J= X_{j_1}\cap\ldots\cap X_{j_n}$. Then there exists a spectral sequence 
$$
E_1^{p,q}=\bigoplus_{\norm{J}=p} H^q_\rig(X_J/\ekd,\sh{E})\Rightarrow H^{p+q}_\rig(X_J/\ekd,\sh{E}).
$$
\end{lem}

\begin{proof} Follows from Lemma 4.4 of \cite{rclsf1}.
\end{proof}

\begin{proof}[Proof of Theorem \ref{finitecurves}] By Lemma \ref{spectral} and the corresponding result for rigid cohomology over $\ek$ (i.e. the spectral sequence resulting from Proposition 2.1.8 of \cite{iso}), the question is local on $X$, and hence we may assume that there exists a finite \'{e}tale map to $\A^1_{k\lser{t}}$. After making a finite separable extension of $k\lser{t}$ we may assume therefore that we have a pushforward functor 
$$ f_{*}:F\text{-}\mathrm{Isoc}^\dagger(X/\ekd) \rightarrow F\text{-}\mathrm{Isoc}^\dagger(\A^1_{k\lser{t}}/\ekd)
$$
as in Theorem \ref{summ}. Hence we may reduce to the case of affine 1-space, i.e. Theorem \ref{finitesheaf}.
\end{proof} 

\begin{defn} A $\varphi$-module over $\ekd$ is a finite dimensional $\ekd$-vector space $M$, together with a Frobenius structure, that is a $\sigma$-linear map $$\varphi:M\rightarrow M$$ which induces an isomorphism $M\otimes_{\ekd,\sigma}\ekd\cong M$.
\end{defn}

\begin{cor} Let $X$ be a smooth curve over $k\lser{t}$, and $\sh{E}\in F\text{-}\mathrm{Isoc}^\dagger(X/\ekd)$. Then the cohomology groups $H^i_\rig(X/\ekd,\sh{E})$ are $\varphi$-modules over $\ekd$.
\end{cor}

\begin{proof} The linearised Frobenius 
$$\varphi^\sigma : H^i_\rig(X/\ekd,\sh{E}) \otimes_{\ekd,\sigma}\ekd\rightarrow H^i_\rig(X/\ekd,\sh{E})$$
becomes an isomorphism upon base-changing to $\ek$, it must therefore be an isomorphism.
\end{proof}

\section{Cohomology with compact support and Poincar\'{e} duality for curves}\label{poindu}

Although we cannot prove it at the moment, we fully expect that $\ekd$-valued rigid cohomology forms an `extended' Weil cohomology. By this we mean that together with all the usual axioms for a Weil cohomology, we can define cohomology groups for arbitrary varieties (not necessarily smooth and proper), as well as cohomology with compact support and support in a closed subscheme. All these should all not only be vector spaces, but in fact $(\varphi,\nabla)$-modules over $\ekd$, and there should be versions of the K\"{u}nneth formula and Poincar\'{e} duality.

In full generality this seems somewhat distant (for example, even finite dimensionality appears rather difficult). In this section we start towards this goal: we define cohomology with compact support and prove Poincar\'{e} duality for smooth curves (as $\varphi$-modules). To define cohomology with compact support requires a bit more care than in the case of `usual' rigid cohomology, since we are trying to capture sections having support compact over $k\lser{t}$, not $k\pow{t}$. In fact this subtlety is acknowledged in le Stum's book on rigid cohomology \cite{rigcoh}, where he only defines the relative rigid cohomology with compact support of a frame $$(X,Y,\frak{P})\rightarrow (S,\overline{S},\frak{S})$$ under the assumption that $S=\overline{S}$. Our goal is to define the `relative rigid cohomology with compact supports' of a smooth and proper morphism of frames $$(X,Y,\frak{P})\rightarrow \left(\spec{k\lser{t}},\spec{k\pow{t}},\spf{\cur{V}\pow{t}}\right)$$ and so this problem cannot be side-stepped. 

Let us first treat the case of constant coefficients, so suppose that we have a smooth frame $(X,Y,\frak{P})$ over $\cur{V}\pow{t}$. We set $Z=Y\setminus X$, $Y'=Y\otimes_{k\pow{t}}k\lser{t}$ and $Z'=Y'\setminus X$. Finally we let $W=Y\otimes_{k\pow{t}}k$. Thus we have a diagram of tubes 
$$ \xymatrix{ & ]X[_\frak{P} \ar[dl]_{j'}\ar[d]_j & \\
]Y'[_\frak{P} \ar[r]^k & ]Y[_\frak{P} & ]W[_\frak{P}\ar[l]_h \\
]Z'[_\frak{P}\ar[u]^{i'}\ar[r]^{k'} & ]Z[_\frak{P}\ar[u]^i}
$$
where $j,j',k$ and $k'$ are closed embeddings, $i,i'$ and $h$ are open embeddings and the lower left had square is Cartesian. For any sheaf $\sh{F}$ on $]Y[_\frak{P}$ we define 
$$ \underline{\Gamma}_{]X[_\frak{P}}(\sh{F}):= \ker (k_*k^{-1}\sh{F}\rightarrow  i_*i^{-1}k_*k^{-1}\sh{F} ),
$$
the total derived functor of $\underline{\Gamma}_{]X[_\frak{P}}$ can then be computed as 
$$ \mathbf{R}\underline{\Gamma}_{]X[_\frak{P}}(\sh{K}^*) = \left( k_*k^{-1}\sh{K}^*\rightarrow \mathbf{R}i_*i^{-1} k_*k^{-1}\sh{K}^* \right)[1].
$$
We define the cohomology with compact support of $(X,Y,\frak{P})$ to be
$$ H^i_{c,\rig}((X,Y,\frak{P})/\ekd):= H^i(]Y[_\frak{P},\mathbf{R}\underline{\Gamma}_{]X[_\frak{P}}(\Omega^*_{]Y[_\frak{P}/S_K})).
$$
To perhaps motivate this definition a bit better, or at least better demonstrate the analogy with compactly supported rigid cohomology, let us slightly recast our definition of $\ekd$-valued rigid cohomology. For any sheaf $\sh{F}$ on $]Y'[_\frak{P}$ we define 
$$ j'^\dagger_X\sh{F}:= j'_*j'^{-1}\sh{F}
$$ 
so that 
\begin{align*} H^i_\rig((X,Y,\frak{P})/\ekd)&=H^i(]Y[_\frak{P},j_*j^{-1}\Omega^*_{]Y[_\frak{P}/S_K}) \\
&= H^i(]X[_\frak{P},j'^{-1}k^{-1}\Omega^*_{]Y[_\frak{P}/S_K}) \\
&= H^i(]Y[_\frak{P},j'_*j'^{-1}k^{-1}\Omega^*_{]Y[_\frak{P}/S_K}) \\
&= H^i(]Y[_\frak{P},j'^\dagger_X (k^{-1}\Omega^*_{]Y[_\frak{P}/S_K}))
\end{align*}
and for any sheaf $\sh{F}$ on $]Y'[_\frak{P}$ we define
$$ \underline{\Gamma}_{]X[_\frak{P}}(\sh{F}) = \ker ( \sh{F}\rightarrow i'_*i'^{-1}\sh{F} ).
$$
The total derived functor of $\underline{\Gamma}_{]X[_\frak{P}}$ is therefore given by
$$ \mathbf{R}\underline{\Gamma}_{]X[_\frak{P}}(\sh{K}^*) = \left( \sh{K}^*\rightarrow \mathbf{R}i'_*i'^{-1}\sh{K}^* \right)[-1]
$$ 
and it is straightforward to verify that we have
$$ H^i_{c,\rig}((X,Y,\frak{P})/\ekd) \cong H^i(]Y'[_\frak{P},\mathbf{R}\underline{\Gamma}_{]X[_\frak{P}}(k^{-1}\Omega^*_{]Y[_\frak{P}/S_K})).
$$
These groups satisfy are functorial in two different ways. First, if $(U,Y,\frak{P})\rightarrow (X,Y,\frak{P})$ is a morphism of smooth and proper frames such that the induced maps on $Y$ and $\frak{P}$ are the identity, and $U\rightarrow X$ is an open immersion, then we get an induced morphism
$$ H^i_{c,\rig}((U,Y,\frak{P})/\ekd) \rightarrow H^i_{c,\rig}((X,Y,\frak{P})/\ekd)
$$
which arises from the natural morphism of functors $ \underline{\Gamma}_{]U[_\frak{P}}\rightarrow \underline{\Gamma}_{]X[_\frak{P}}$. Also, if we have a morphism of smooth and proper frames
$$ u:(X',Y',\frak{P}') \rightarrow (X,Y,\frak{P})
$$
such that the diagram
$$\xymatrix{ X' \ar[r]\ar[d] & Y' \ar[d] \\ X \ar[r] &Y }
$$ is Cartesian, then we get an induced morphism
$$ H^i_{c,\rig}((X,Y,\frak{P})/\ekd) \rightarrow H^i_{c,\rig}((X',Y',\frak{P}')/\ekd)
$$ 
arising from the natural morphism of functors $\underline{\Gamma}_{]X[_\frak{P}}u_{K*} \rightarrow u_{K*}\underline{\Gamma}_{]X'[_\frak{P'}}$. One also verifies easily that compactly supported cohomology only depends on a neighbourhood of $]X[_\frak{P}$ inside $]Y[_\frak{P}$, since if $j_V:V\rightarrow ]Y[_\frak{P}$ is the inclusion of such a neighbourhood, then we have a natural isomorphism of functors
$$ j_{V*}\underline{\Gamma}_{]X[_\frak{P}}j_V^{-1} \cong \underline{\Gamma}_{]X[_\frak{P}}
$$ 
where $\underline{\Gamma}_{]X[_\frak{P}}$ on the LHS denotes the obvious analogous functor for sheaves on $V$. The key step in proving that $H^i_{c,\rig}((X,Y,\frak{P})/\ekd)$ only depends on $X$ is to show a Poincar\'{e} lemma with compact supports. First, however, we need an excision sequence.

\begin{lem}[Excision] Let $(X,Y,\frak{P})$ be a smooth frame, and $U\subset X$ an open subvariety. Let $Z$ be a closed complement to $U$ in $X$, $\overline{Z}=Y\setminus U$, and $i:]\overline{Z}[_\frak{P}\rightarrow ]Y[_\frak{P}$ the corresponding open immersion. Then for any sheaf $\sh{F}$ on $]Y[_\frak{P}$ there is an exact triangle
$$ \mathbf{R}\underline{\Gamma}_{]U[_\frak{P}}(\sh{F}) \rightarrow \mathbf{R}\underline{\Gamma}_{]X[_\frak{P}}(\sh{F}) \rightarrow \mathbf{R}i_*\mathbf{R}\underline{\Gamma}_{]Z[_\frak{P}}(i^{-1}\sh{F}) \overset{+1}{\rightarrow}
$$
of sheaves on $]Y[_\frak{P}$, and hence a long exact sequence
$$ \ldots\rightarrow H^i_{c,\rig}((U,Y,\frak{P})/\ekd)\rightarrow H^i_{c,\rig}((X,Y,\frak{P})/\ekd) \rightarrow H^i_{c,\rig}((Z,\overline{Z},\frak{P})/\ekd)\rightarrow \ldots
$$
in compactly supported cohomology.
\end{lem}

\begin{proof} First note that the natural map $Z\rightarrow \overline{Z}$ is an open immersion, so that $(Z,\overline{Z},\frak{P})$ is indeed a smooth frame. Let $Y'=Y\otimes_{k\pow{t}}k\lser{t}$ as usual, and also let $\overline{Z}'=Y'\cap \overline{Z}$, $W=Y\setminus X$. Consider the diagram of tubes
$$ \xymatrix{ ]Y'[_\frak{P}\ar[r]^k & ]Y[_\frak{P} \\ ]\overline{Z}'[_\frak{P} \ar[r]^{k'} \ar[u]^{i'} & ]\overline{Z}[_\frak{P} \ar[u]^i & ]W[_\frak{P}\ar[l]_j \ar[ul]_{i_W}
}
$$ where the square is Cartesian, $k,k'$ are closed immersions, $i,i',j,i_W$ are open immersions, also note that we have $W=\overline{Z}\setminus Z$. Thus we have
\begin{align*} \mathbf{R}\underline{\Gamma}_{]U[_\frak{P}}(\sh{F}) &\simeq \left(k_*k^{-1}\sh{F}\rightarrow \mathbf{R}i_*i^{-1}k_*k^{-1}\sh{F}\right)[-1] \\
\mathbf{R}\underline{\Gamma}_{]X[_\frak{P}}(\sh{F}) &\simeq \left(k_*k^{-1}\sh{F}\rightarrow \mathbf{R}i_{W*}i_W^{-1}k_*k^{-1}\sh{F}\right)[-1] \\
\mathbf{R}i_*\mathbf{R}\underline{\Gamma}_{]Z[_\frak{P}}(i^{-1}\sh{F}) &\simeq \left(\mathbf{R}i_*k'_*k'^{-1}i^{-1}\sh{F}\rightarrow \mathbf{R}i_*\mathbf{R}j_*j^{-1}k'_*k'^{-1}i^{-1}\sh{F}\right)[-1]
\end{align*}
and hence it suffices to show that there is a quasi-isomorphism
\begin{align*} \mathbf{R}i_*i^{-1}k_*k^{-1}\sh{F} \rightarrow \mathbf{R}i_{W*}i_W^{-1}k_*k^{-1}i^{-1}\sh{F}  \\ \simeq  \mathbf{R}i_*k'_*k'^{-1}\sh{F}\rightarrow \mathbf{R}i_*\mathbf{R}j_*j^{-1}k'_*k'^{-1}i^{-1}\sh{F} .
\end{align*}
But unpacking this all, setting $\sh{G}=i^{-1}k_*k^{-1}\sh{F}$, and noting that we have $k'_*i'^{-1}\cong i^{-1}k_*$, then both terms in this simply become
$$  \mathbf{R}i_*\sh{G}\rightarrow \mathbf{R}i_*\mathbf{R}j_*j^{-1}\sh{G} 
$$
which proves the first claim. The second them immediately follows.
\end{proof}

\begin{lem}[Poincar\'{e} Lemma with compact supports] Let $(X,Y,\frak{P})$ be a smooth frame over $\cur{V}\pow{t}$, and consider the morphism of frames
$$ u:(X,Y,\widehat{\A}^d_\frak{P})\rightarrow (X,Y,\frak{P}).
$$
Then
$$ \mathbf{R}\underline{\Gamma}_{]X[_\frak{P}}( \Omega^*_{]Y[_\frak{P}/S_K} ) \rightarrow \mathbf{R}u_{K*}\mathbf{R}\underline{\Gamma}_{]X[_{\widehat{\A}^d_\frak{P}}}( \Omega^*_{]Y[_{\widehat{\A}^1_\frak{P}}/S_K} )
$$
is a quasi-isomorphism.
\end{lem}

\begin{proof} The question is local on $\frak{P}$, which we may thus assume to be affine, and by considering further localisations, it suffices to prove the statement after applying the derived global sections functor. As in the proof of the excision theorem (but much simpler) we have a quasi-isomorphism
$$ \mathbf{R}\Gamma(]Y[_{\widehat{\A}^d_\frak{P}}, \mathbf{R}\underline{\Gamma}_{]X[_{\widehat{\A}^d_\frak{P}}}( \Omega^*_{]Y[_{\widehat{\A}^d_\frak{P}}/S_K}) \simeq \left(\mathbf{R}\Gamma_\rig ((Y',Y,\widehat{\A}^d_\frak{P})/\ekd) \rightarrow \mathbf{R}\Gamma_\rig ((Z,\overline{Z},\widehat{\A}^d_\frak{P})/\ekd) \right)[-1] 
$$
where $Y'=Y\otimes_{k\pow{t}} k\lser{t}$, $Z=Y'\setminus X$ and $\overline{Z}$ is the closure of $Z$ inside $Y$, and, for example,
$$\mathbf{R}\Gamma_\rig ((Y',Y,\widehat{\A}^d_\frak{P})/\ekd)= \mathbf{R}\Gamma(]Y[_{\widehat{\A}^d_\frak{P}},j_{Y'}^\dagger\Omega^*_{]Y[_{\widehat{\A}^d_\frak{P}}/S_K} )$$
is the cohomology complex computing the rigid cohomology of the frame $(Y',Y,\widehat{\A}^d_\frak{P})$. Similarly, we have 
$$ \mathbf{R}\Gamma(]Y[_\frak{P}, \mathbf{R}\underline{\Gamma}_{]X[_\frak{P}}( \Omega^*_{]Y[_\frak{P}/S_K}) \simeq \left(\mathbf{R}\Gamma_\rig ((Y',Y,\frak{P})/\ekd) \rightarrow \mathbf{R}\Gamma_\rig ((Z',\overline{Z'},\frak{P})/\ekd) \right)[-1] .
$$
But the Poincar\'{e} Lemma without supports (Proposition 4.3 of \cite{rclsf1}) shows that the maps
\begin{align*} \mathbf{R}\Gamma_\rig ((Y',Y,\frak{P})/\ekd) &\rightarrow \mathbf{R}\Gamma_\rig ((Y',Y,\widehat{\A}^d_\frak{P})/\ekd) \\
\mathbf{R}\Gamma_\rig ((Z,\overline{Z},\frak{P})/\ekd) &\rightarrow  \mathbf{R}\Gamma_\rig ((Z,\overline{Z},\widehat{\A}^d_\frak{P})/\ekd)
\end{align*}
are quasi-isomorphisms (see also Remark 4.7 of \emph{loc. cit.}) and the result follows.
\end{proof}

\begin{cor} Let $u:(X,Y',\frak{P}')\rightarrow (X,Y,\frak{P})$ be a smooth and proper morphism of smooth frames, with $u^{-1}X\cap Y' = X$. Then the natural map
$$ H^i_{c,\rig}((X,Y,\frak{P})/\ekd) \rightarrow H^i_{c,\rig}((X,Y',\frak{P}')/\ekd)
$$
is an isomorphism.
\end{cor}

\begin{proof} Taking into account the Strong Fibration Theorem, that $H^i_{c,\rig}((X,Y,\frak{P})/\ekd)$ only depends on a neighbourhood of $]X[_\frak{P}$ in $]Y[_\frak{P}$, and that we may localise on $X$ using the excision sequence, this then follows exactly as in the proof of Theorem 4.5 of \cite{rclsf1}. 
\end{proof}

Hence we get the following.

\begin{theo} Up to natural isomorphism, $H^i_{c,\rig}((X,Y,\frak{P})/\ekd)$ only depends on $X$ and not on the choice of smooth and proper frame $(X,Y,\frak{P})$, thus it makes sense to write $H^i_{c,\rig}(X/\ekd)$. These compactly supported cohomology groups are covariant with respect to open immersions and contravariant with respect to proper morphisms.
\end{theo}

\begin{proof} The only thing that is not clear is functoriality with respect to proper morphisms, but this just follows from the fact that if we have a diagram
$$ \xymatrix{ X' \ar[r]\ar[d] & Y' \ar[d] \\ X \ar[r] & Y}
$$
with both vertical maps proper, and both horizontal maps open immersions, and $X'$ dense inside $Y'$, then the diagram is actually Cartesian.
\end{proof}

For non-constant coefficients $\sh{E}\in\mathrm{Isoc}^\dagger(X/\ekd)$, we perform an entirely similar construction, but simply replacing $]Y[_\frak{P}$ by some suitable neighbourhood of $]X[_\frak{P}$ to which $\sh{E}$ extends. Specifically, if $V$ is some neighbourhood of $]X[_\frak{P}$ inside $]Y[_\frak{P}$, and $\sh{F}$ is a sheaf of $V$, then denoting by $i_V$ the inclusion of the complement of $]X[_\frak{P}$ into $V$ and by $k_V$ the inclusion of $V':= V\cap ]Y'[_\frak{P}$ into $V$, we define
$$ \underline{\Gamma}_{]X[_\frak{P}}(\sh{F}) = \ker (  k_{V*}k_V^{-1}\sh{F}\rightarrow i_{V*}i_V^{-1}k_{V*}k_V^{-1}\sh{F})
$$
as above. If $\sh{E}\in\mathrm{Isoc}^\dagger(X/\ekd)$ extends to a module with connection $E_V$ on some such $V$, we define
$$ H^i_{c,\rig}((X,Y,\frak{P})/\ekd,E_V) := H^i(V,\mathbf{R}\underline{\Gamma}_{]X[_\frak{P}}(E_V\otimes \Omega^*_{V/S_K}))
$$
exactly as before. Of course, there is an entirely similarly defined functor $ \underline{\Gamma}_{]X[_\frak{P}} $ for sheaves on $V'$ we have 
$$H^i_{c,\rig}((X,Y,\frak{P})/\ekd,E_V) \cong H^i(V',\mathbf{R} \underline{\Gamma}_{]X[_\frak{P}}(k_V^{-1}(E_V\otimes \Omega^*_{V/S_K})) )$$
as before. This does not depend on the choice of neighbourhood $V$, or on the extension $E_V$ of $\sh{E}$, we will therefore write it as $H^i_{c,\rig}((X,Y,\frak{P})/\ekd,\sh{E}) $

\begin{rem}
If we define the endofunctor $j_X'^\dagger$ for sheaves on $V'$ in the obvious way, then we also have
$$ H^i_\rig(X/\ekd,\sh{E}) \cong H^i(V',j_X'^\dagger (k_V^{-1}(E_V\otimes\Omega^*_{V/S_K})))
$$
so the cohomology without supports of $\sh{E}$ can also be computed in terms of sheaves on $V'$.
\end{rem}

\begin{prop} This does not depend on the choice of smooth and proper frame $(X,Y,\frak{P})$ containing $X$, in that if we have a smooth and proper morphism
$$ u:(X,Y',\frak{P}')\rightarrow (X,Y,\frak{P})
$$
of smooth frames, such that $u^{-1}(X)\cap Y' = X$, then 
$$H^i_{c,\rig}((X,Y',\frak{P}')/\ekd,\sh{E})\rightarrow H^i_{c,\rig}((X,Y,\frak{P})/\ekd,\sh{E})
$$
is an isomorphism.
\end{prop}

\begin{proof} This is similar to the constant coefficients case, although there is a slight subtlety in the extension of the Poincar\'{e} Lemma to include coefficients. In the constant case we knew we could extend to a module with connection on the whole of the tube $]Y[_\frak{P}$, in general we will only be able to extend to some neighbourhood $V$ of $]X[_\frak{P}$. We get round this as follows. 

First we note that it certainly suffices to prove the stronger statement that
$$ \mathbf{R}\underline{\Gamma}_{]X[_\frak{P}}( \sh{E}\otimes \Omega^*_{]Y[_\frak{P}/S_K} ) \rightarrow \mathbf{R}u_{K*}\mathbf{R}\underline{\Gamma}_{]X[_{\widehat{\A}^d_\frak{P}}}( u^*\sh{E}\otimes \Omega^*_{]Y[_{\widehat{\A}^1_\frak{P}}/S_K} )
$$ is a quasi-isomorphism. We use the fact that this question is local on $\frak{P}_K$ to base change to some $[Y]_n$, and hence assume that $]Y[_\frak{P}=\frak{P}_K$, which we may also assume to be affinoid. We then use an excision sequence to reduce to the case when $X=Y\cap D(g)$ for some $g\in \cur{O}_\frak{P}$, thus we may assume that the neighbourhood $V$ is affinoid, and thus has a formal model. Again using the fact that compactly supported cohomology only depends on some neighbourhood of $]X[_\frak{P}$, we may further base change to this formal model of $V$ to then ensure that we do get a module with connection on the whole of $]Y[_\frak{P}$.
\end{proof}

Thus it makes sense to write these groups as $ H^i_{c,\rig}(X/\ekd,\sh{E})$, they are functorial in $\sh{E}$, as well as being covariant with respect to open immersions $X\rightarrow X'$ when $\sh{E}$ extends to $X'$ and contravariant with respect to proper morphisms. We also get the functorialities with respect to finite separable extensions of $k\lser{t}$ as well as Frobenius as in \S5 of \cite{rclsf1} (we will come back to the issue of the extension $\ekd\rightarrow \ek$ shortly). Hence, as in the case of cohomology without supports, if $\sh{E}\in F\text{-}\mathrm{Isoc}^\dagger(X/\ekd)$ we get a Frobenius on  $H^i_{c,\rig}(X/\ekd,\sh{E})$, that is a $\sigma$-linear morphism
$$ H^i_{c,\rig}(X/\ekd,\sh{E})\rightarrow H^i_{c,\rig}(X/\ekd,\sh{E}),
$$
however, we do not know if this is a $\varphi$-module structure (i.e. linearises to an isomorphism) in general.

Since for any sheaf $\sh{F}$ on some open neighbourhood $V'$ of $]X[_\frak{P}$ inside $]Y'[_\frak{P}$ there is an obvious map 
$$ \underline{\Gamma}_{]X[_\frak{P}}\sh{F} \rightarrow j'^\dagger_X\sh{F}
$$ which is trivially an isomorphism when $X=Y'$, for any embeddable variety there is a natural `forget supports' map
$$ H^i_{c,\rig}(X/\ekd,\sh{E})\rightarrow H^i_\rig(X/\ekd,\sh{E})
$$
which is an isomorphism when $X$ is proper. This is compatible with Frobenius when $\sh{E}\in F\text{-}\mathrm{Isoc}^\dagger(X/\ekd)$. We also have an excision sequence
$$ \ldots\rightarrow H^i_{c,\rig}(U/\ekd,\sh{E}|_U)\rightarrow H^i_{c,\rig}(X/\ekd,\sh{E}) \rightarrow H^i_{c,\rig}(Z/\ekd,\sh{E}|_Z)\rightarrow \ldots
$$
for any open immersion $U\subset X$ of embeddable varieties over $k\lser{t}$ with complement $Z$, and any $\sh{E}\in\mathrm{Isoc}^\dagger(X/\ekd)$.

The next problem is to relate $H^i_{c,\rig}(X/\ekd,\sh{E})$ to the `usual' rigid cohomology with compact supports $H^i_{c,\rig}(X/\ek,\hat{\sh{E}})$, or at least produce a canonical base change morphism
$$  H^i_{c,\rig}(X/\ekd,\sh{E}) \otimes_{\ekd}\ek \rightarrow H^i_{c,\rig}(X/\ek,\hat{\sh{E}}).
$$
To do so, we suppose that we have a smooth and proper frame $(X,Y,\frak{P})$, and we let $(X,Y',\frak{P}')$ denote the base change of this frame to $\cur{O}_{\ek}$. Let $Z$ be the complement of $X$ in $Y$, and $Z'$ the complement of $X$ in $Y'$. For simplicity we will stick to the constant coefficients case, the general case can be handled by replacing $]Y[_\frak{P}$ by a suitable neighbourhood of $]X[_\frak{P}$. We consider the diagram
$$\xymatrix{ ]Z[_\frak{P} \ar[r]^i & ]Y[_\frak{P} & ]Y'[_\frak{P} \ar[l]_k \\
]Z'[_{\frak{P}'}\ar[r]^{i'} \ar[u] & ]Y'[_{\frak{P}'} \ar[u]_j \ar[ur] }
$$
so that we have 
$$\mathbf{R}\underline{\Gamma}_{]X[_\frak{P}}(\sh{F}) \cong \left( k_*k^{-1}\sh{F}\rightarrow \mathbf{R}i_*i^{-1}k_*k^{-1}\sh{F} \right)[-1]$$
for any sheaf $\sh{F}$ on $]Y[_\frak{P}$. Define the functor $\underline{\Gamma}_{]X[_{\frak{P}'}} = \ker \sh{F}\rightarrow i'_*i'^{-1}\sh{F}$ for sheaves on $]Y'[_{\frak{P}'}$, so that we have
$$ \mathbf{R}\underline{\Gamma}_{]X[_\frak{P}'}(\sh{F}) \cong \left(\sh{F}\rightarrow \mathbf{R}i'_*i'^{-1}\sh{F} \right)[-1].
$$
Then we have a natural isomorphism of functors
$$ j^{-1}\mathbf{R}\underline{\Gamma}_{]X[_\frak{P}} \cong \mathbf{R}\underline{\Gamma}_{]X[_\frak{P}'}j^{-1}
$$
which therefore induces an $\ekd$-linear morphism
$$ H^i_{c,\rig}(X/\ekd) \rightarrow H^i(]Y'[_{\frak{P}'}, \mathbf{R}\underline{\Gamma}_{]X[_{\frak{P}'}}(\Omega^*_{]Y'[_{\frak{P}'}/\ek}))
$$
and hence an $\ek$-linear morphism 
$$ H^i_{c,\rig}(X/\ekd)\otimes_{\ekd}\ek \rightarrow H^i(]Y'[_{\frak{P}'}, \mathbf{R}\underline{\Gamma}_{]X[_{\frak{P}'}}(\Omega^*_{]Y'[_{\frak{P}'}/\ek})).
$$
It therefore suffices to show that the RHS computes compactly supported $\ek$-valued rigid cohomology of $X$. But since the underlying topoi of a Tate rigid space and the corresponding adic space locally of finite type over $\ek$ coincide, and this equivalence is functorial, it follows that the complex of sheaves $\mathbf{R}\underline{\Gamma}_{]X[_{\frak{P}'}}(\Omega^*_{]Y'[_{\frak{P}'}})$ on $]Y'[_{\frak{P}'}$ gives rise to the same object in the underlying topos of $]Y'[_{\frak{P}'}$ as the analogous construction given by Berthelot in \cite{rigcohber} in terms of Tate's rigid spaces. Hence we have 
$$ H^i(]Y'[_{\frak{P}'}, \mathbf{R}\underline{\Gamma}_{]X[_{\frak{P}'}}(\Omega^*_{]Y'[_{\frak{P}'}/\ek})) \cong H^i_{c,\rig}(X/\ek)
$$
as expected, and this gives us our base change morphism. For smooth curves, and with coefficients which extend to the compactification, we can easily see that this is an isomorphism as follows. 

\begin{lem} Let $X/k\lser{t}$ be a smooth curve with compactification $\overline{X}$ and suppose that $\sh{E}\in F\text{-}\mathrm{Isoc}^\dagger(X/\ekd)$ extends to an overconvergent $F$-isocrystal $\overline{\sh{E}}$ on $\overline{X}$. Then the base change morphism
$$ H^i_{c,\rig}(X/\ekd,\sh{E})\otimes_{\ekd}\ek \rightarrow H^i_{c,\rig}(X/\ek,\hat{\sh{E}})
$$
is an isomorphism.
\end{lem}

\begin{proof} Of course, since smooth curves are always embeddable, $H^i_{c,\rig}(X/\ekd,\sh{E})$ is actually defined. Let $D$ be the closed complement of $X$. Since cohomology with and without supports agree for $\overline{X}$ and $D$ we have an exact sequence
\begin{align*} 0 \rightarrow H^0_{c,\rig}(X/\ekd,\sh{E})&\rightarrow H^0_\rig(\overline{X}/\ekd,\overline{\sh{E}})\rightarrow H^0_\rig(D/\ekd,\overline{\sh{E}}|_D) \\ &\rightarrow H^1_{c,\rig}(X/\ekd,\sh{E}) \rightarrow H^1_{\rig}(\overline{X}/\ekd,\overline{\sh{E}})\rightarrow 0
\end{align*}
as well as an isomorphism
$$ H^2_{c,\rig}(X/\ekd,\sh{E}) \cong H^2_\rig(\overline{X}/\ekd,\overline{\sh{E}}).
$$
 Of course, there same holds over $\ek$, and this exact sequence and isomorphism are compatible with base change. Since we know that base change holds for $H^i_\rig(\overline{X}/\ekd,\overline{\sh{E}})$ and $H^i_\rig(D/\ekd,\overline{\sh{E}}|_D)$, it follows that it also must hold for $H^i_{c,\rig}(X/\ekd,\sh{E})$.
\end{proof}

Actually, thanks to excision we now get finite dimensionality and base change for compactly supported cohomology for any curve, without smoothness hypotheses.

\begin{cor} Let $X$ be a curve over $k\lser{t}$. Then the base change morphism
$$ H^i_{c,\rig}(X/\ekd)\otimes_{\ekd}\ek \rightarrow H^i_{c,\rig}(X/\ek)
$$
is an isomorphism.
\end{cor}

\begin{proof} Since every variety becomes generically smooth after a (possibly inseparable) finite extension of $k\lser{t}$, it suffices to prove that $H^i_{c,\rig}(X/\ekd,\sh{E})$ commutes with finite extensions of $k\lser{t}$. The case of separable extensions is handled entirely similarly to Lemma \ref{basefinite}, it therefore suffices to treat the case of the extension $k\lser{t}\rightarrow k\lser{t}^{1/p}$. If we let $\cur{E}_K^{\dagger,\sigma^{-1}}$ denote a copy of $\ekd$ but with $\ekd$-algebra given by the Frobenius lift $\sigma$. This makes $\cur{E}_K^{\dagger,\sigma^{-1}}$ into a finite extension of $\ekd$, and hence we can apply methods entirely similar to those used in the proof of Lemma \ref{basefinite}.
\end{proof}

Hence, as usual, if $X/k\lser{t}$ is a smooth curve, and $\sh{E}\in F\text{-}\mathrm{Isoc}^\dagger(X/\ekd)$ extends to a compactification $\overline{X}$, then $H^i_{c,\rig}(X/\ekd,\sh{E})$ is a $\varphi$-module over $\ekd$ (that is, it is finite dimensional and the linearised Frobenius is an isomorphism). If $X$ is singular, then we at least know that the same is true for $H^i_{c,\rig}(X/\ekd)$.

Now, if $X/k\lser{t}$ is any Frobenius-embeddable variety, i.e. there exists smooth and proper frame $(X,Y,\frak{P})$ which admits a lifting of Frobenius, then we can construct a Frobenius-compatible Poincar\'{e} pairing as follows. Suppose that $X$ is purely of dimension $n$. Since for any sheaves $\sh{F},\sh{G}$ of $\ekd$-modules on some neighbourhood $V'$ of $]X[_\frak{P}$ in $]Y'[_\frak{P}$ there is a natural isomorphism
$$ \cur{H}\mathrm{om}_{\ekd}(j'^\dagger_X\sh{F},\sh{G})\cong \underline{\Gamma}_{]X[_\frak{P}}\cur{H}\mathrm{om}_{S_K}(\sh{F},\sh{G}),
$$
we easily deduce the existence of a pairing 
$$ H^i_\rig(X/\ekd,\sh{E})\times H^j_{c,\rig}(X/\ekd,\sh{F})\rightarrow H^{i+j}_{c,\rig}(X/\ekd,\sh{E}\otimes \sh{F})
$$
for any $\sh{E},\sh{F}\in\mathrm{Isoc}^\dagger(X/\ekd)$. Hence we get a pairing
$$ H^i_\rig(X/\ekd,\sh{E})\times H^{2n-i}_{c,\rig}(X/\ekd,\sh{E}^\vee)\rightarrow H^{2n}_{c,\rig}(X/\ekd)
$$ 
where $\sh{E}^\vee$ is the dual of $\sh{E}$, which is Frobenius-compatible when $\sh{E}\in F\text{-}\mathrm{Isoc}^\dagger(X/\ekd)$.

\begin{theo} Suppose that $X$ is a smooth curve. Then there is a trace morphism
$$\mathrm{Tr}:H^{2}_{c,\rig}(X/\ekd)\rightarrow \ekd(-1)$$ of $\varphi$-modules over $\ekd$, which is an isomorphism when $X$ is geometrically irreducible. Moreover, for any $\sh{E}\in F\text{-}\mathrm{Isoc}^\dagger(X/\ekd)$ which extends to some $\overline{\sh{E}}$ on a compactification $\overline{X}$ of $X$, the induced Poincar\'{e} pairing
$$ H^i_\rig(X/\ekd,\sh{E})\times H^{2-i}_{c,\rig}(X/\ekd,\sh{E}^\vee)\rightarrow \ekd(-1)
$$
is perfect.
\end{theo}

\begin{rem} \begin{enumerate} \item Here $\ekd(-1)$ means the constant $\pn$-module, but with the Frobenius action multiplied by $q$. 
\item Since a smooth curve is always Frobenius-embeddable, the statement of the theorem makes sense.
\end{enumerate}
\end{rem}

\begin{proof} Let us first suppose that $X$ is geometrically irreducible. Then upon base-changing to $\ek$, $H^2_{c,\rig}(X/\ekd)$ becomes one-dimensional and the Poincar\'{e} pairing perfect, the same is therefore true over $\ekd$. Thus it suffices to show that $H^2_{c,\rig}(X/\ekd)\cong \ekd(-1)$, by excision we may assume that $X$ is proper, and hence admits a finite map to $\P^1$. Thus we get a natural map
$$ H^2_{\rig}(\P^1_{k\lser{t}}/\ekd) \rightarrow H^2_{\rig}(X/\ekd)
$$
which becomes an isomorphism upon base changing to $\ek$, and is therefore an isomorphism. Thus it suffices to treat the case $X=\P^1_{k\lser{t}}$. This can be computed directly.

In general, there exists a finite, Galois extension $F/k\lser{t}$ over which $X$ breaks up into a disjoint union of geometrically irreducible components. Hence by summing the trace maps on each component we get
$$ \mathrm{Tr}:H^2_{c,\rig}(X_F/\cur{E}_K^{\dagger,F}) = H^2_{c,\rig}(X/\cur{E}_K^{\dagger})\otimes_{\ekd} \cur{E}_K^{\dagger,F} \rightarrow \cur{E}_K^{\dagger,F}(-1)
$$
and hence using the trace map $\cur{E}_K^{\dagger,F}\rightarrow \ekd$ (since $\cur{E}_K^{\dagger,F}/\ekd$ is a finite Galois extension), we get an induced trace map
$$ \mathrm{Tr}:H^2_{c,\rig}(X/\cur{E}_K^{\dagger})\rightarrow \ekd(-1).
$$
That the induced pairing is perfect then follows again from base change to $\ek$.
\end{proof}

\bibliographystyle{amsplain}\addcontentsline{toc}{section}{References}
\bibliography{/Users/cdl10/Documents/Dropbox/Maths/lib.bib}

\end{document}